\documentclass[12pt,reqno]{bryant-article}

\title[Ideal Triangulations and Once-Punctured Surface Bundles]{Ideal Triangulations and Once-Punctured Surface Bundles}
\author[B. Bryant]{Birch Bryant}
\address{University of North Alabama}
\email{bbryant3@una.edu}
\date{}

\usepackage{graphicx}
\usepackage{lipsum}

\begin{document}

	\begin{abstract}
		A well-known result of Walsh states that if $\mathcal T^*$ is an ideal triangulation of an atoroidal, acylindrical, irreducible, compact 3-manifold with torus boundary components and $\mathcal T^*$ has essential edges, then every properly embedded, two-sided, incompressible surface $S$ is isotopic to a spun-normal surface in $\mathcal T^*$ unless $S$ is isotopic to a fiber or virtual fiber. For a given manifold $M$ that fibers over $S^1$, it was previously unknown whether there exists an ideal triangulation in which the fiber appears as a spun-normal surface. We prove that such a triangulation exists and give an algorithm to construct the ideal triangulation provided $M$ has a single boundary component.
		
	\end{abstract}
	\maketitle
	
	\section{Introduction}
	Normal and spun-normal surface theory lies at the heart of algorithmic 3-manifold topology. Since Haken’s foundational work on unknot recognition \cite{Haken}, normal surfaces have played a central role in key algorithmic advances, including Thompson’s \cite{thompsonthin} and Rubinstein’s \cite{RubinsteinAlmostNormal} solutions to the 3-sphere recognition problem, and the detection of reducible surgeries on knots by Jaco, Rubinstein, and Sedgwick \cite{Findingplanarsurfaces}. Spun-normal surfaces have proven highly effective in computational topology. They connect deeply with topological quantum invariants, such as the 3D Index, as computed by Garoufalidis, Hodgson, Hoffman, and Rubinstein \cite{3dindex}, and the colored Jones polynomial, studied by Futer, Kalfagianni, and Purcell \cite{Gutsbook}. Especially relevant to this work is the algorithm of Cooper, Tillmann, and Worden \cite{CoopTW}, which utilizes embedded and immersed spun-normal surfaces to compute the unit Thurston norm ball for cusped hyperbolic 3-manifolds with torus boundary.
	
	Following the groundbreaking results of Bergeron and Wise \cite{BergeronWise} in 2012 and Agol \cite{agolvirtual} in 2013, it is known that every hyperbolic 3-manifold (in our case with torus cusps) is virtually fibered—that is, it admits a finite cover which fibers over the circle (we include fibers in this notion of virtual fibers). The virtual fibers are carried by faces of the Thurston norm ball \cite{TNB}. Yet virtual fibers are at issue with spun-normal surface theory. A well-known theorem by Walsh \cite{walsh} illustrates this obstacle: in a triangulation with essential edges, every properly embedded, two-sided, incompressible, $\partial$-incompressible surface that is not a virtual fiber is isotopic to a spun-normal surface. This highlights a gap: virtual fibers typically fall outside the reach of spun-normal surface techniques. Cooper, Tillmann, and Worden address this issue directly in their work (Question 5.1 of \cite{CoopTW}), asking whether every fibered, once-cusped, hyperbolic 3-manifold admits an ideal triangulation in which the fiber is isotopic to an embedded spun-normal surface. The aim of this paper is to give an affirmative answer to that question and to present an algorithm for constructing such a triangulation.
	
	In Section \ref{Sec:2} we give the definitions of triangulations, ideal triangulations, normal, and spun-normal surfaces used throughout this work. In Subsection \ref{sec:lwtaut} we recall key results from Tollefson and Wang \cite{TollefsonWang} in the particular case of fibered 3-manifolds. This culminates in Lemma \ref{vertexfiber}, which finds a fiber amongst a finite set of normal surfaces in a triangulation. 
	\vspace{1ex}
	\begin{lemn}[\ref{vertexfiber}]
		Let $M_K$ be the exterior of a knot in a rational homology 3-sphere which fibers over the circle with fiber $F$. For any triangulation $\mathcal T$ of $M_K$ there exists an lw-taut vertex solution $F'$ isotopic to $F$.\vspace{1ex}
	\end{lemn}
	
	Section \ref{Sec:3} discusses two complementary operations on triangulations: crushing along normal surfaces and inflating ideal triangulations. Crushing triangulations along normal surfaces is a technique developed by Jaco and Rubinstein in \cite{Jaco2003} to reduce the complexity of 3-manifolds. A review of this process is given in Subsection \ref{sec:crush}. Subsection \ref{sec:inflate} covers the inverse process--inflating ideal triangulations--developed by  Jaco and Rubinstein in \cite{Jaco14}. Particular attention is paid here to the boundary of the inflation triangulations and how the push-off of the boundary intersects individual tetrahedra. In Subsection \ref{sec:short}, we introduce the notion of \textit{short inflations}, those having a minimal number of inflated tetrahedra. This minimality plays a central role in our approach and, when combined with the concept of \textit{compatible} normal surfaces, forms the foundation for the two main technical results of this work:
	
	\vspace{1ex}
	\begin{thmn}[\ref{shortquads}]
		Let $M$ be an orientable, compact, irreducible, $\partial$-irreducible 3-manifold with non-empty connected boundary homeomorphic to a torus. Let $S$ be a connected, properly embedded, incompressible, $\partial$-incompressible surface in $M$ with $\partial S$ connected. There exists a short inflation triangulation $\mathcal T$ of $M$ in which $S$ is isotopic to a compatible normal surface. Further, there is an algorithm to construct such a $\mathcal T$.
	\end{thmn}
	
	\vspace{1ex}
	
	\begin{thmn}[\ref{crushsurface}]
		Let $M$ be an orientable, compact, irreducible, $\partial$-irreducible 3-manifold with non-empty connected boundary homeomorphic to a torus. Let $\mathcal T$ be a short inflation with $S$ a compatible normal surface. The ideal triangulation $\mathcal T^*$ obtained by crushing $\mathcal T$ along $\partial M$ has a spun-normal surface $S^*$ with normal quadrilaterals completely determined by the normal quadrilaterals of $S$ in $\mathcal T$, the lifts of which (relative to the crushing map) are contained in $\widetilde{\mathbf{\Delta}}^*\subset\widetilde{\mathbf{\Delta}}$. Further, $S^*$ represents the interior of a properly embedded surface in $\mathring M$ isotopic to $\mathring S$.
	\end{thmn} 
	
	\vspace{1ex}
	
	Section \ref{Sec:Algo} contains the main theorems of this work, algorithms built over the course of the preceding sections that construct ideal triangulations in which a fiber is isotopic to a spun-normal surface and identifies such a spun-normal surface amongst a finite set of surfaces. The  first result applies to knots in rational homology 3-spheres:
	
	\vspace{1ex}
	
	\begin{thmn}[\ref{main1}]
		Let $M_K$ be the exterior of a knot in a rational homology 3-sphere, and suppose $M_K$ fibers over the circle with fiber $F$. There is an algorithm to construct an ideal triangulation $\mathcal T^*$ of $M_k$ with $S$ a spun-normal surface such that $\mathring S$ is isotopic to $\mathring F$ in $\mathring M_K$. Further, the algorithm will identify $S$.
	\end{thmn}
	
	\vspace{1ex}
	\noindent The second result gives an algorithm for constructing an ideal triangulation of a once-cusped hyperbolic 3-manifold $M$ and identifies the fiber as isotopic to a spun-normal surface of the form $\sum n_i F_i$ where $F_i$ are vertex solutions and $\sum n_i \leq \beta_1$ where $\beta_1$ is the first Betti number of $M$. This construction not only realizes the fiber as a spun-normal surface but also gives a quantitative bound on its complexity. In doing so, it offers a complete and constructive answer to the question posed by Cooper, Tillmann, and Worden.
	
	\vspace{1ex}
	
	\begin{thmn}[\ref{main2}]
		Let $M$ be a compact, connected, irreducible, $\partial$-irreducible, atoroidal, orientable 3-manifold with non-empty connected boundary homeomorphic to a torus. Suppose $M$ fibers over the circle with fiber $F$. There is an algorithm to construct an ideal triangulation $\mathcal T^*$ of $M$ with $S$ a spun-normal surface such that $\mathring S$ is isotopic to $\mathring F$. Further the algorithm will identify $S$.
	\end{thmn}

	\vspace{1ex}
	Section \ref{Sec:Examples} gives two examples of the implementation of Theorem \ref{main1}: first for the complement of the trefoil knot, and second, for the complement of the (-2,3,7)-pretzel knot. In both cases, the ideal triangulations constructed are, to the best of the author's knowledge, the first in the literature with the property that the fiber ($S_{1,1}$ and $S_{5,1}$ resp.) is isotopic to an embedded spun-normal surface. The section concludes with a set of open questions aimed at extending the study of such triangulations.
	
	\subsection*{Acknowledgments} This work began as the author's PhD dissertation, supervised by Neil Hoffmann. The author would like to thank him for his patience and continued support. The author would like to thank William Jaco for initiating this investigation. The author would like to express their gratitude to Saul Schleimer for key references and encouraging the focus on vertex solutions over fundamental solutions. The author gratefully acknowledges that part of the work and code used to generate examples in this paper were developed with support from Henry Segerman’s NSF grant (DMS-1708239, Three- and Four-Dimensional Triangulations and Mathematical Visualization, 2017–2021) during a summer research assistantship at Oklahoma State University.  Finally, the author would like to thank Ian Agol and Joel Hass for stimulating conversations leading to some of the specific examples in Section \ref{Sec:Examples}.
	\section{Triangulations and Normal Surface Theory}\label{Sec:2}
	
	We adopt the notation used in \cite{Jaco2003} for triangulations and ideal triangulations. Let $\widetilde{\mathbf{\Delta}} = \{\widetilde\Delta_i\}_i$ be a collection of pairwise disjoint tetrahedra, $\mathbf{\Phi}$ a collection of orientation-reversing face identifications, or \textit{glueings}. The identification space $X=\widetilde{\mathbf{\Delta}}\slash \mathbf{\Phi}$ is a 3-manifold at each point except possibly at a collection of vertices, $\{v_i\}$. If $X=M$ is a manifold, we say $\mathcal T=(\widetilde{\mathbf{\Delta}},\mathbf{\Phi})$ is a \textit{triangulation} of $M$. If $X$ is not a manifold, $X\backslash\{v_i\}$ is homeomorphic to the interior of a compact 3-manifold with boundary, $M$, and we say $\mathcal T^*=(\widetilde{\mathbf{\Delta}},\mathbf{\Phi})$ is an \textit{ideal triangulation} of $M$ and we refer to the tetrahedra of $\mathcal T^*$ as \textit{ideal tetrahedra}. We denote the image of the $k$-subsimplices of $\widetilde{\mathbf{\Delta}}$ under the quotient map by $\mathcal{T}^{(k)}$ (or ${\mathcal{T}^*}^{(k)}$).

Throughout this paper, $M$ will be an orientable, compact, irreducible, $\partial$-irreducible 3-manifold with non-empty connected boundary homeomorphic to a torus. By work of Jaco and Rubinstein \cite{Jaco2003}, we can assume that an ideal triangulation, $\mathcal T^*$ of $M$, has a single vertex, $v^*$, referred to as the \textit{ideal vertex}. The frontier of a regular  neighborhood $N(v^*)$ we call the \textit{ideal boundary} of $\mathcal T^*$. For a triangulation $\mathcal T$ of $M$ it is necessary for some tetrahedra to have unglued faces. These unglued faces triangulate $\partial M$, we refer to this decomposition of $\partial M$ as the \textit{material boundary} of $\mathcal T$.

In our construction, the tetrahedra of $\widetilde{\mathbf{\Delta}}$ are not necessarily embedded in $X$, but their interiors are always embedded. For each tetrahedron $\Delta$ in $X$, there is exactly one tetrahedron $\widetilde{\Delta}\in\widetilde{\mathbf{\Delta}}$, that projects to $\Delta$, we refer to $\widetilde{\Delta}$ as the \textit{lift} of $\Delta$. A face, $\sigma$ in $X$ lifts to two faces in $\widetilde{\mathbf{\Delta}}$ (not necessarily on distinct tetrahedra) if and only if the interior of $\sigma$ is in the interior of $X$. Otherwise $\sigma$ lifts to exactly one face in $\widetilde{\mathbf{\Delta}}$ and $\sigma$ is in the material boundary of $\mathcal T$.

For the 3-manifolds of interest in this work, ideal-triangulations can be constructed using numerous methods. Bing \cite{Bing} gives a method to construct a triangulation of $M$ having material boundary. This material boundary can be collapsed to an ideal vertex using the methods of Jaco and Rubinstein \cite{Jaco2003}. In \cite{Lackenby}, Lackenby gives an algorithm to construct a taut ideal triangulation, the faces of which carry all properly embedded, $\chi_-$-minimizing incompressible surfaces.
	
	\subsection{Normal and Spun-Normal Surface Theory}\label{Norm}
	
	We give here a brief overview of the details from normal surface theory that will be used in this work.

Fix a tetrahedron $\widetilde \Delta \in \widetilde{\mathbf{\Delta}}$. A \textit{normal arc} in a 2-simplex $\sigma$ of $\widetilde \Delta$ is a properly embedded arc with each boundary component lying in a distinct edge of $\sigma$. A \textit{normal disc} in $\widetilde \Delta$ is a properly embedded triangle or quadrilateral with boundary edges in distinct faces of $\widetilde \Delta$. There are four distinct isotopy classes of normal triangles and three distinct isotopy classes of normal quadrilaterals, these are each referred to as \textit{normal disc types}.

Let $\mathcal T$ be a triangulation of $M$. Let $S$ be a properly embedded surface in $M$ which is transverse to $\mathcal{T}^{(2)}$, we say that $S$ is \textit{normal} (\textit{normal in} $\mathcal T$) if for each tetrahedron $\Delta$ of $\mathcal T$, the intersection of $S$ with $\Delta$ lifts to a finite collection of normal triangles and normal quadrilaterals in $\widetilde{\Delta}$. An isotopy of $M$ is called a \textit{normal isotopy} (with respect to $\mathcal T$) if it is invariant on each simplex of $\mathcal T$. A triangulation $\mathcal T$ is said to have \textit{normal boundary} if the frontier of a regular neighborhood of $\partial M$ is normally isotopic to a normal surface in $\mathcal T$. We call said normal surface a \textit{boundary-linking surface}. It is not necessary in general for a triangulation with material boundary to have normal boundary, but it will be so for the triangulations constructed later in this work.

For a given triangulation, $\mathcal T$, with $t = |\mathcal T^{(3)}|$ tetrahedra there are $7t$ normal isotopy classes of normal discs. Given an ordering of these normal disc types $d_1,d_2, \dots d_{7t}$ we can represent a given normal surface $S$ by the vector $(x_1,x_2,\dots,x_{7t})\in \mathbb R^{7t}_{\geq 0}$ where $x_i$ is the number of normal discs in the induced cell-decomposition of $S$ of type $d_i$. Suppose $S$ meets a face $\sigma$ of $\mathcal T$ in a normal arc $\gamma$. For each of $\widetilde{\Delta}_i,\widetilde{\Delta}_j$ (possible $i=j$) identified at $\sigma$, $\gamma$ determines one normal triangle type and one normal quadrilateral type. Let $d_i,d_j,d_k,d_l$ be those types, then the face identification of $\mathcal T$ determines a \textit{matching equations}: $x_i+x_j=x_k+x_l$. For each face $\sigma$ in $\mathcal T$ that lifts to two faces of some $\widetilde{\Delta}_i,\widetilde{\Delta}_j\in \widetilde{\mathbf{\Delta}}$, there are three such equations determined by the three normal arc classes in $\sigma$.  Non-negative integer solutions to this system give a parametrization of the normal isotopy classes of normal surfaces in $\mathcal T$. We add the restriction that the solutions to the matching equations have non-negative coordinates and we get a cone in the positive-orthant of $\mathbb R^{7t}_{\geq 0}$. We refer to this cone as the \textit{solution space}, $\mathcal{S}_{\mathcal T}$. We refer to the integer lattice points of $\mathcal{S}_\mathcal{T}$ as \textit{solutions}. For a normal surface $S$, we refer to both the embedding and the normal isotopy class as simply $S$. If in addition to the gluing equations we require $\sum x_i=1$, we determine a compact, convex linear cell. The rational points in this cell correspond to the projective classes of points in  $\mathcal S_{\mathcal T}$. We refer to this compact, convex linear cell as the \textit{projective solution space} denoted $\mathcal{P}_{\mathcal T}$. A rational point in $\mathcal P_\mathcal{T}$ is an \textit{admissible} solution if corresponding to each tetrahedra there is at most one of the normal quadrilateral types whose corresponding coordinate is non-zero. We denote the projective class of a normal surface $S$ by $\vec{S}\in \mathcal{P}_\mathcal{T}$. Every admissible solution in $\mathcal{P}_\mathcal{T}$ is the projective class of some normal surface in $\mathcal T$.

In the other direction, it is natural to wonder what conditions are required for a surface to be isotopic to some normal surface in $\mathcal T$. Of use for this work we have the following result due to Haken:

\begin{theorem}[\cite{Haken}]\label{haken}
	Let $M$ be a compact, irreducible 3-manifold with triangulation $\mathcal T$. If $S$ is an incompressible, $\partial$-incompressible, properly embedded surface in $M$, then $S$ is isotopic to a normal surface in $\mathcal T$.
\end{theorem}

The \textit{carrier} $C_S$ of a normal surface $S$ is the unique minimal face of $\mathcal{P}_\mathcal{T}$ that contains $\vec{S}$.  A normal surface $F$ is said to be \textit{carried} by $C_S$ if $\vec{F}\in C_S$. We have that $F$ is carried by $C_S$ if and only if there exists a normal surface $G$ such that $nS=F+G$ for some positive integer $n$. Every rational point of $C_S$ is an admissible solution. In particular, a vertex $\vec{v}$ of $C_S$ (more generally a vertex of $\mathcal{P}_\mathcal{T}$) is a rational point and must be admissible as each coordinate corresponding to a normal disc type not in $S$ is zero. We say a normal surface is a \textit{vertex solution} if its parametrization is $k\vec{v}$ where $k$ is the smallest integer for which $k\vec{v}$ has non-negative integer coordinates.

Let $\mathcal T^*$ be an ideal triangulation of $M$. We here adapt the notation from \cite{Kang2005}. Let $S$ be a closed compact properly embedded surface in $M$, transverse to $\mathcal{T}^{(2)}$. As before, we say $S$ is a normal surface if for each ideal tetrahedron $\Delta$ of $\mathcal T^*$, the intersection of $S$ with $\Delta$ lifts to a finite collection of normal triangles and normal quadrilaterals. The ideal boundary of $\mathcal T^*$ is normally isotopic to a normal surface, called the \textit{vertex-linking surface}, consisting of a single normal triangle of each type. 

If $S$ is a compact surface with boundary, properly embedded in $M$, then the interior, $\mathring S$, is properly embedded in $\mathring M$. Suppose for each ideal tetrahedron of $\mathcal T^*$, the intersection of $\mathring S$ lifts to a collection of normal discs. The intersection of the closure of a small regular neighborhood $N$ of the vertex $v$ of $\mathcal T^*$ with the interior of $S$ is a half-open annulus, $A = \mathring S \cap \bar N$, consisting of an infinite number of normal triangles and possibly some pieces of normal triangles at $\partial A$. We say $A$ is an \textit{infinite trivial normal annulus}. This annulus ``spins" around the ideal boundary of $\mathcal T^*$ and we say $S$ is a \textit{spun-normal surface}. Thus a spun-normal surface is characterized by having a finite number of normal quadrilaterals and an infinite number of normal triangles of each normal isotopy class. 

Truncating the ideal vertex of $\mathcal T^*$ produces a cell-decomposition of $M$ consisting of truncated tetrahedra with hexagonal faces identified. The unglued triangular faces form a triangulation of $\partial M$. The intersection of $\mathring S$ with a truncated tetrahedron is a collection of quadrilaterals, triangles, truncated triangular and truncated quadrilateral discs as depicted in Figure \ref{fig:disctypes}. These discs form a cell-decomposition on the properly embedded surface $S$ that we call a \textit{core of the spun-normal surface} \cite{KRSpun}. The number of such discs is dependent on the truncation, but all such decompositions yield a surface $\bar S$ isotopic to $S$ in $M$. As such they are all cores for the spun-normal surface. The \textit{boundary-slope} of the spun-normal surface $S$ is the homotopy class of a non-nullhomotopic component of the intersection of $\bar S$ with the unglued triangular faces of the truncated tetrahedra cell-decomposition of $M$. As $\partial M$ is a torus, there can only be one such homotopy class for a properly embedded surface.

\begin{figure}
	\includegraphics[width=4.5in]{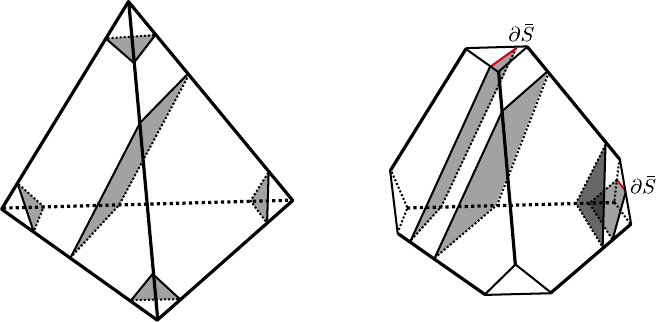}\caption{Left: Various normal disc types in a tetrahedron. Right: discs in a truncated tetrahedron. The surface boundary shown in red.}\label{fig:disctypes}
\end{figure}

In \cite{TollefsonQ}, Tollefson showed that for a triangulation $\mathcal T$ the admissible solutions are completely determined by their normal quadrilaterals. Kang \cite{Kang2005} extended this result to show for ideal triangulations, both normal and spun-normal surfaces are completely determined by their normal quadrilaterals. In both instances, the notation will be as follows: given a (ideal) triangulation with $t$ tetrahedra, there are $3t$ normal isotopy classes of normal quadrilaterals. Fix an ordering of these normal quadrilateral types $d_1^Q, d_2^Q, \dots, d_{3t}^Q$. There is a unique representation of a (spun-)normal surface by the vector $S_Q=(q_1,q_2,\dots, q_{3t})\in \mathbb R^{3t}_{\geq 0}$. Consider an edge $e_k=\langle ab\rangle$  with positive orientation from $a$ to $b$ in ($\mathcal T^*$) $\mathcal T$. Let $d^Q$ be a normal quadrilateral type in $\widetilde{\Delta}$ and let $v_j$ be a corner of a normal quadrilateral of type $d^Q$ incident to $\hat{e}_k=\langle \hat a \hat b \rangle$ an edge of $\widetilde{\Delta}$ in the lift of $e_k$. Each disc of type $d^Q$ has within its boundary, a normal arc separating vertex $\hat a$ in a face of $\widetilde\Delta$ containing $\hat e_k$ and a normal arc separating vertex $\hat b$ in the other face containing $\hat e_k$. If for the disc type $d^Q$ a right handed twist around $\hat e_k$ moves from the face of $\widetilde{\Delta}$ containing the normal arc of $d^Q$ separating $\hat b$ to the face containing the normal arc separating $\hat a$, we define the \emph{sign} of the corner $v_j$ with respect to the edge $e_k$ to be $\delta_{k,j}=+1$. If a right hand twist around $\hat e_k$ moves from the arc $\hat a$ to $\hat b$, then we set $\delta_{k,j}=-1$. For every corner $v_j$ that is incident to no lift $\hat e_k$ the sign is defined to be $\delta_{k,j}=0$. For a normal quadrilateral of type $d^Q$ we define the \textit{sense} $\epsilon_k$ relative to the edge $e_k$ to be the total of the signs of each of the four corners of a quadrilateral of the given type. Consider for a (spun-)normal surface $S$ all normal quadrilateral types incident to some edge $e_k$ in (ideal) triangulation $\mathcal T$. The number of normal quadrilaterals in the cell decomposition of $S$ of positive sense must be equal to the number of those of negative sense. The parametrization $S_Q=(q_1,q_2,\dots, q_{3t})\in \mathbb R^{3t}_{\geq 0}$ is a solution to the linear system of equations, the \emph{Q-matching equations} with one equation for each edge $e_k$ in $\mathcal T$:\begin{equation}\label{eq:qmatch}
	\bigg\{ \sum^{3t}_{i=1}\epsilon_{k,i}q_i=0\bigg\}_k
\end{equation}
	
	\subsection{Fibers and \textit{lw}-Taut Surfaces}\label{sec:lwtaut}
	
	Let $F$ be a compact, orientable, once-punctured surface with $h\in \text{Aut}(F)$. Let $M_h$ be the mapping torus $M_h = F \times [0,1]\slash \sim$, where $(x,0)\sim(h(x),1)$. We say a 3-manifold $M$ \textit{fibers over the circle} $S^1$ with \emph{fiber} $F$ if $M$ is homeomorphic to $M_h$ for some $h\in \text{Aut}(F)$. This gives a proper embedding of $F$ in $M$ such that the image is incompressible and $\partial$-incompressible. The identification of $M$ with the mapping torus has strong implications on the homology groups $H_2(M)$ and $H_2(M,\partial M)$. In particular, any incompressible surface properly embedded in $M$ that is homologous to $F$ is homotopic into $F$ \cite{neuwirth}. In the case of knots in rational homology 3-spheres this relation is stronger.

Let $M$ be a rational homology 3-sphere, $K$ a null-homologous knot in $\Sigma$. Let $M_K=M \backslash \overline{N(K)}$ be the knot exterior in $M$. Suppose further that $M_K$ fibers of the circle with fiber $F$. Then $F$ is a connected surface with a single boundary component and the fiber bundle $F\times [0,1]\slash \sim$ induces a fiber bundle on the torus boundary of $M$. The relative homology group $H_2(M,K)\cong H_1(K)$ is generated by the homology classes $[K]$. As $[\partial F]=[K]$ we can identify the relative homology class of the fiber $[F] = [K]$ through the isomorphism $H_2(M_K,\partial M_K)\cong H_2(M,N(K))\cong H_2(M,K)$. It can then be completely determine if a surface, $S$, is isotopic to a fiber by examining $\partial S$.

\begin{lemma}[\cite{Edmonds,Kitayama2022}]\label{fibiso}
	Let $M$ be a 3-manifold that fibers over the circle with fiber $F$, a once-punctured orientable surface. If $S \subset M$ is a connected, properly embedded, two-sided, incompressible, $\partial$-incompressible surface with $[S]=[F]\in H_2(M,\partial M;\mathbb Z)$, then $S$ is isotopic to $F$.
\end{lemma}

By Theorem \ref{haken}, a fiber of a bundle structure $F$ is isotopic to a normal surface in a given triangulation $\mathcal T$. We will denote this normal surface simply as $F$. It is standard in applications of normal surface theory to identify a surface being studied as a vertex solution or give a way to construct said surface as the sum of vertex solutions. To this end we discuss the work by Tollefson and Wang \cite{TollefsonWang} on taut surfaces. 

For a properly embedded surface $S$ in $M$, define  $$\chi_-(S)=-\chi(S\backslash\{\text{spherical and disc components of }S\}).$$An oriented, incompressible, $\partial$-incompressible surface, $S$, properly embedded in $M$ is said to be \textit{taut} if the homology class $[S]\in H_2(M,\partial M)$ is non-trivial, there is no homologically trivial union of components of $S$, and $S$ is $\chi_-$-minimizing over all properly embedded surfaces representing $[S]$. If for a triangulation $\mathcal T$, $S$ is transverse to the 2-skeleton of $\mathcal T$, we define the \textit{weight} of $S$ by $wt(S)=|S\cap \mathcal T^{(0)}|$. We say $S$ is \textit{lw-taut} if it has minimal weight among all taut surfaces representing the homology class $[S]$. Define $kS$ to be the disjoint union of $k$ copies of $S$. If $S$ is lw-taut it follows that $kS$ is an lw-taut surface with homology class $k[S]$. The following from \cite{TollefsonWang} gives existence of normal surfaces that are lw-taut:

\begin{lemma}[Lemma 2.3 in \cite{TollefsonWang}]\label{tautnorm}
	Let $S$ be a taut surface in a triangulated 3-manifold $M$ such that $S$ is transverse to $\mathcal T$. Then there exists a taut normal surface $G$ homologous to $S$ such that $wt(G)\leq wt(S)$.
\end{lemma}

For a normal surface $S$, we say the carrier $C_S$ is \textit{taut} (\textit{lw-taut}) if every surface carried by $C_S$ is taut (lw-taut). The study of  faces of $\mathcal{P}_\mathcal T$  that carry only minimal weight normal surfaces began with Jaco and Oertel \cite{JacoOertel} which led to an algorithm to detect if a 3-manifold was Haken (contains an essential surface of genus $\geq$ 1). Again from \cite{TollefsonWang} we have the existence of such faces:

\begin{lemma}[Theorem 3.3 in \cite{TollefsonWang}]\label{tautface}
	Let $S$ be an oriented, lw-taut, normal surface. Then $C_S$ is lw-taut and there are unique orientations assigned ot the surfaces of $C_S$ such that if $G,H$ are carried by $C_S$ then $[G+H]=[G]+[H]$.
\end{lemma}

A useful property for minimal faces is Theorem 5.1 from Jaco and Segdwick \cite{Jaco98}. The result follows immediately from Theorem 6.5 of Jaco and Tollefson \cite{jacoalgorithms}.

\begin{proposition}[Theorem 5.1 in \cite{Jaco98}]\label{prop:jacosedgwick}
	Let $\mathcal T$ be a triangulation of an irreducible 3-manifold $M$. If $S$ is a two-sided, incompressible, $\partial$-incompressible, normal surface in $\mathcal T$ with least weight, then every rational point in $C_S$ is the projective class of an embedded, incompressible, $\partial$-incompressible, two-sided, normal surface in $\mathcal T$.
\end{proposition}

If a knot exterior $M_K$ fibers over the circle, Lemma \ref{fibiso} guarantees the fiber $F$ is taut. We  then have a special case Corollary 4.2 of \cite{TollefsonWang}, finding a vertex solution isotopic to $F$.

\begin{lemma}\label{vertexfiber}
	Let $M_K$ be the exterior of a knot in a rational homology 3-sphere which fibers over the circle with fiber $F$. For any triangulation $\mathcal T$ of $M_K$ there exists an lw-taut vertex solution $F'$ isotopic to $F$.
\end{lemma}

\begin{proof}
	Let $F$ be the fiber of given bundle structure on $M_K$, and let $\mathcal T$ be a triangulation of $M_K$. By Lemma \ref{tautnorm} there exists an lw-taut normal surface isotopic to $F$, which we will simply refer to again as $F$. Let $C_{[F]}$ be the minimal lw-taut face of $\mathcal{P}_\mathcal{T}$ that carries every lw-taut normal surface representing $[F]$. As every homology class in $[C_{[F]}]$ is of the form $m[F]$, there must be some lw-taut vertex solution $F_i$ such that $[F_i]=[F]$. By Proposition \ref{prop:jacosedgwick}, $F_i$ is two-sided, incompressible, and $\partial$-incompressible. Hence by Lemma \ref{fibiso}, $F_i$ is isotopic to $F$.
\end{proof}

\begin{remark}\label{thurstonvert}
	More generally, if $M$ fibers over the circle and the fiber, $F$, is a once-punctured orientable surface, then we will need to construct a normal surface representing $F$ by a sum of vertex solutions. Using the techniques of Thurston \cite{TNB} and Corollary 5.8 of Tollefson and Wang \cite{TollefsonWang}, there is a collection of vertex solutions $\{F_i\}_{i\leq \beta_1}$, where $\beta_1$ is the first Betti number of $M$, such that $\sum F_i$ is isotopic to $F$.
\end{remark}
	
	\section{Crushing and Inflating Triangulations}\label{Sec:3}

	\subsection{Crushing Triangulations along Normal Surfaces}\label{sec:crush}
	The general notion of crushing a triangulation along a normal surface was introduced in \cite{Jaco2003}. The particular triangulations and normal surfaces to which we will be interested in applying this technique behave exceptionally nicely and we omit most details that are not immediately applicable. It is noted that the following applies to both triangulations and ideal triangulations of any compact 3-manifold. Thus there is no issue in restricting to those 3-manifolds of interest to this work. Further, we will omit references to ideal triangulations in the proceeding definitions. See \cite{Jaco14} for a more detailed description of inflations.

Let $\mathcal T$ be a triangulation of $M$. Suppose $S$ is a closed normal surface in $\mathcal T$ and $X$ the closure of a component of the complement of $S$ that does not contain the vertex of $M$. The triangulation $\mathcal T$ induces a cell-decomposition $\mathcal C_X$ on $X$ consisting of four distinct cell types: \textit{truncated-tetrahedra}, \textit{truncated-prisms}, \textit{triangular blocks}, and \textit{quadrilateral blocks}.

We say a 1- or 2-cell of $\mathcal C_X$ is \textit{horizontal} if their interior is in $S$ and \textit{vertical} otherwise. The quadrilateral vertical 2-cells are called \textit{trapezoids}. There are two trapezoids in a truncated-prism, three in  a triangular block, and four in a rectangular block. All other vertical 2-cells are hexagons found in truncated-prisms and truncated-tetrahedra. Define $\mathbb P(\mathcal C_X)=\{\text{vertical edges of }\mathcal C_X\}\cup\{\text{trapezoids}\}\cup \{\text{triangular blocks}\}\cup \{\text{quadrilateral blocks}\}$. If each component $\mathbb P_i$ of $\mathbb P(\mathcal C_X)$ is a product $I$ bundle, then we call $\mathbb P(\mathcal C_X)$ a \textit{combinatorial product} and write $\mathbb P_i=K_i \times I$ where $K_i \times 0$ and $K_i \times 1$ are isomorphic subcomplexes of $S$. If each $K_i$ is simply connected we say $\mathbb P(\mathcal C_X)$ is a \textit{trivial combinatorial product}.

A hexagonal face of a truncated-prism is identified to a hexagonal face of a truncated-tetrahedron or of a (possibly the same) truncated-prism such that the horizontal edges of the faces are identified if and only if the vertices that their lifts separate in the corresponding tetrahedron $\widetilde{\Delta}$ are identified in $\mathcal T$. Define for each hexagonal face of $\mathcal C_X$ a midpoint and an arc connecting the two midpoints in each truncated-prism. The graph $\Gamma$ defined by these arcs and midpoints has components homeomorphic to intervals and copies of $S^1$. The intervals have endpoints in hexagonal faces of truncated-tetrahedra. If each component $\Gamma$ is homeomorphic to an interval we say $\mathcal C_X$ is \textit{simple}. We call the collection of truncated-prisms containing each component of $\Gamma$ a \textit{chain}.

If $C_X$ is simple and $\mathbb P(\mathcal C_X)$ is a trivial combinatorial product, we can construct an ideal triangulation of $X$ by \textit{crushing} the cells of $\mathcal C_X$ as follows: each component of $S$ is crushed to a point, components of $\mathbb P(\mathcal C_X)$ are crushed to edges in the ideal triangulation, truncated-prisms are crushed to faces, and truncated-tetrahedra become ideal tetrahedra, see Figure \ref{fig:crushingmap}.

\begin{figure}
	\includegraphics[width=4.8in]{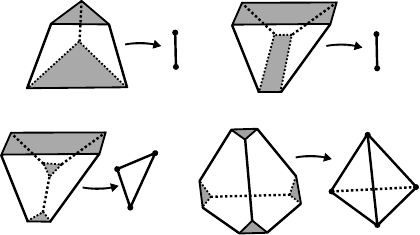}\caption{Shown are the results of crushing each cell type of $\mathcal C_X$. The gray faces represent cells of the normal surface, the rounded vertices represent the ideal vertex of $\mathcal T^*$.}\label{fig:crushingmap}
\end{figure}

It is left only to give the family of face identifications, $\mathbf \Phi^*$ of the ideal tetrahedra. Notice a tetrahedron of $\mathcal T = (\widetilde{\mathbf{\Delta}},\mathbf{\Phi})$ can contain at most one truncated-tetrahedra. Further as $C_X$ is simple, a tetrahedron can have at most one of a truncated-tetrahedron or a truncated-prism. Write the subset of tetrahedra containing a truncated-tetrahedra as $\widetilde{\mathbf{\Delta}}^* = \{\widetilde{\Delta}^*_1,\widetilde{\Delta}^*_2,\dots,\widetilde{\Delta}^*_n\}\subset \widetilde{\mathbf \Delta}$. If two hexagonal faces of truncated-prisms are identified with corresponding identification $\phi \in \mathbf \Phi$, then $\phi \in \mathbf \Phi^*$. Let $C$ be a chain of truncated-prisms and $\Gamma_C$ the corresponding component of the graph $\Gamma$ previously described. The intersection $C \cap S$ has three components, each of which has a single normal arc in the faces of some $\widetilde{\Delta}^*_i,\widetilde{\Delta}^*_j$ (possibly $i=j$) corresponding to the endpoints of $\Gamma_C$. Suppose the lift of the normal arcs separate pairs of vertices $(a_i,a_j)$, $(b_i,b_j)$, and $(c_i,c_j)$ in  $\widetilde{\Delta}^*_i,\widetilde{\Delta}^*_j$ respectively. Then for faces $\sigma_i = \langle a_i b_i c_i\rangle,\sigma_j = \langle a_j b_j c_j\rangle$, define $\phi_C \in \mathbf{\Phi}^*$ by $\phi:\sigma_i \rightarrow \sigma_j$ defined by the rules on the vertices $\phi_C(a_i)=a_j,\phi_C(b_i)=b_j,\phi_C(c_i)=c_j$. Then every face in $\widetilde{\mathbf{\Delta}}^*$ is identified by some map in $\mathbf{\Phi}^*$. We call $\mathcal T^* = (\widetilde{\mathbf{\Delta}}^*, \mathbf{\Phi}^*)$ the ideal triangulation obtained by \textit{crushing} $\mathcal T$ \textit{along} $S$. We use the following version of Theorem 4.1 found in \cite{Jaco2003}:

\begin{theorem}[\cite{Jaco2003}]\label{FTC}
	Suppose $\mathcal T$ is a triangulation of a compact, orientable 3-manifold $M$. Suppose $S$ is a closed normal surface in $\mathcal T$, $X$ is the closure of a component of the complement of $S$ containing no vertex of $\mathcal T$. If $\mathcal C_X$ is simple and $X \neq \mathbb P(\mathcal C_X)$ is a trivial combinatorial product, then $\mathcal T$ can be crushed along $S$ and the triangulation $\mathcal T^* = (\widetilde{\mathbf{\Delta}}^*, \mathbf{\Phi}^*)$ is an ideal triangulation of $X$.
	
\end{theorem}

If one crushes all cells of $\mathcal C_X$ except for the truncated tetrahedra, the resulting cell-complex is the same cell-decomposition of $X$ described in Section \ref{Norm} used to compute the boundary-slope of spun-normal surfaces. Attaching the cone over the unglued triangular faces of this cell-complex recovers the ideal triangulation $\mathcal T^*$ of $X$ in Theorem \ref{FTC} above.
	
	\subsection{Inflating Triangulations}\label{sec:inflate}
	
	Let $\mathcal T$ be a triangulation of a 3-manifold $M$. We say $\mathcal T$ is a \textit{minimal-vertex triangulation} if for any triangulation $\mathcal T'$ of $M$ $|\mathcal T^{(0)}|\leq |\mathcal T'^{(0)}|$. If $\partial M$ is connected, then a minimal-vertex triangulation has exactly one vertex \cite{Jaco2003}, inducing a 1-vertex triangulation on $\partial M$; when $\partial M $ is homeomorphic to a torus, this is the standard two-triangle triangulation of the torus.

If $\mathcal T$ has normal boundary and can be crushed along the boundary-linking surface, we say $\mathcal T$ can be \textit{crushed along}  $\partial M$. Let $\mathcal T^*$ be an ideal triangulation of $M$, $\mathcal T$ a triangulation of $M$. We say $\mathcal T$ is an \textit{inflation of} $\mathcal T^*$ if $\mathcal T$ has normal boundary and can be crushed along $\partial M$ such that the ideal triangulation obtained by crushing along $M$ is isomorphic to $\mathcal T^*$. More generally if $\mathcal T$ has the above properties we say $\mathcal T$ is an \textit{inflation triangulation} when we do not wish to emphasize $\mathcal T^*$.

For a given an ideal triangulation, Jaco and Rubinstein \cite{Jaco14} give an algorithm to construct an inflation triangulation. We will briefly discuss here the process of adding tetrahedra to the ideal triangulation and the characterization of each tetrahedron in relation to the boundary-linking surface. We will refer the reader to \cite{Jaco14} for the full algorithm.

For a surface $S$, a \textit{spine} is an embedded 1-complex such that each component of its complement is an open disc. Given a triangulation of $S$, a subcomplex $\xi$ of the 1-skeleton is called a \textit{frame} if it is a spine of $S$ and is minimal among all spines of $S$ relative to set inclusion. Notably, the complement of a frame is connected and for the torus there are, up to homeomorphism, two possible frames.

A vertex of a frame is called a \textit{branch point} if it has index (in the frame) greater than two. The closure of a component of the complement of a branch point is called a \textit{branch}. Given a frame, $\xi$, of a triangulated torus there is either one branch point of index four or two branch points of order 3. Further, there are either two branches or three branches respectively.

Given an ideal triangulation $\mathcal T^*$ of $M$, recall in the context of this work $\mathcal T^*$ has a single ideal vertex. Let $S$ be the vertex-linking surface in $\mathcal T^*$ and let $\xi$ be a frame of $S$ in the triangulation given by the normal triangles. The inflation begins with the tetrahedra in $\widetilde{\mathbf{\Delta}}^*$. Given an edge $\varepsilon$ of $\xi$ in a 2-simplex $\sigma$ of $\mathcal T^*$, we unglue the two tetrahedra at $\sigma$ and attach a new tetrahedron $\widetilde b_\varepsilon$ called a \textit{band tetrahedron} as in Figure \ref{fig:edgeinflate}. If there are multiple edges of $\xi$ in $\sigma$ we repeat this process. If for a given edge, $e$ of $\mathcal T^*$ there are edges of $\xi$ ending at $e$ such that their projection along $e$ onto a regular disc neighborhood of the corresponding vertex in $S$ has a crossing, then we add a new tetrahedron $\widetilde c_e$ called a \textit{crossing tetrahedron}. Finally, for each branch point, $v$, we add $index(v)-2$ new tetrahedra $\widetilde p_1,\widetilde p_2$ called \textit{branch-point tetrahedra}. Face identifications are determined as in \cite{Jaco14} so that $\mathcal T_\xi=(\widetilde{\mathbf{\Delta}}_\xi,\Phi)$ is a minimal vertex triangulation of $M$, where $\widetilde{\mathbf{\Delta}}_\xi = \widetilde{\mathbf{\Delta}}^* \sqcup \{\widetilde b_\varepsilon\} \sqcup \{\widetilde c_e\} \sqcup \{\widetilde p_i\}$. The triangulation $\mathcal T_\xi$ has normal boundary and admits a crushing along $\partial M$ yielding $\mathcal T^*$.

The lift of the cells of the boundary-linking surface $B$ meet the tetrahedra of $\widetilde{\mathbf{\Delta}}_\xi$ in a distinguishable way. Let $X\cong M$ be the closure of the component of the complement of $B$ not containing the vertex of $\mathcal T_\xi$. Each tetrahedron of $\widetilde{\mathbf{\Delta}}_\xi$ contains exactly one 3-cell of the induced cell-decomposition $\mathcal C_X$. Namely, for $\widetilde{\Delta}^* \in \widetilde{\mathbf{\Delta}}^*$, $\widetilde{\Delta}^*$ contains a truncated-tetrahedron, $\widetilde b_\varepsilon$ contains a truncated-prism, $\widetilde c_e$ contains a quadrilateral block, and $\widetilde p_i$ contains a triangular block. The chains of truncated-prisms have length at most three. For each edge $\varepsilon$ in $\xi$ not incident to a branch point, the face of a band tetrahedron $\widetilde b_\varepsilon$ containing the trapezoidal face of the truncated-prism is glued to another band tetrahedron or some $\widetilde c_e$ such that the normal arc corresponding to the normal quadrilateral glues to the normal arc of a normal quadrilateral in the to be identified tetrahedron. If $\widetilde b_\varepsilon$ corresponds to an edge ending at a branch point, the appropriate face containing the trapezoidal face of the truncated-prism glues to a face of some $\widetilde p_i$ or if necessary some $\widetilde c_e$ which is in turn glued to $\widetilde p_i$. Exactly one face of each $\widetilde p_i$ remains unglued, together forming the material boundary of $\mathcal T_\xi$.

\begin{remark}\label{naming}
	The above description can be used to identify a given triangulation $\mathcal T = (\widetilde{\mathbf{\Delta}},\mathbf{\Phi})$ of $M$ as the inflation of some ideal triangulation $\mathcal T^*$ of $M$. In particular, we can write $\widetilde{\mathbf{\Delta}} = \widetilde{\mathbf{\Delta}}^* \sqcup \{\widetilde b_i\} \sqcup \{\widetilde c_j\} \sqcup \{\widetilde p_k\}$, where $B,X$ are given as above. Here  $\widetilde{\mathbf{\Delta}}^*$ is the collection of tetrahedra containing a truncated-tetrahedron of $\mathcal C_X$, $\{\widetilde b_i\}$ the collection containing a truncated-prism, $\{\widetilde c_j\}$ the collection containing a quadrilateral block, $\{\widetilde p_k\}$ the collection containing a triangular block. Finally, $\mathcal C_X$ must be simple; from the inflation procedure a chain can contain at most three truncated prisms.
\end{remark} 

We refer to the latter three sets as \textit{band tetrahedra, crossing tetrahedra}, and \textit{branch-point tetrahedra}. The band tetrahedra form subcomplexes of $\mathcal T$ defined by their identifications at faces containing trapezoids of $\mathcal C_X$. If $\mathcal T$ contains no crossing tetrahedra, these subcomplexes are in 1-1 correspondence with the branches of some frame of the associated ideal triangulation obtained by crushing $\partial M$. In this case we refer to the subcomplexes of band tetrahedra again as \textit{branches}. The frame can be recovered by the relation of the boundary-linking surface and vertex-linking surface after crushing as depicted in Figure \ref{fig:edgeinflate}. 

\begin{figure}
	\includegraphics[width=4.3in,height=2in]{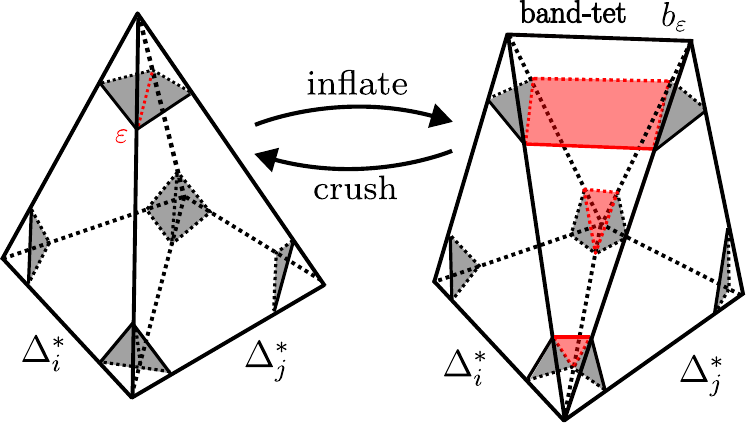}\caption{The step of inflating at the edge $\varepsilon$ of frame $\xi$. Note the truncated prism in $b_\varepsilon$ and the hexagonal cells in the faces common to $b_\varepsilon$ and $\Delta^*_i,\Delta^*_j$.}\label{fig:edgeinflate}
\end{figure}
	
	\subsection{Short Inflations}\label{sec:short}
	
	Given an inflation triangulation $\mathcal T_\xi$ of $M$, we are interested in constructing a triangulation $\mathcal T'$ with minimal number of band tetrahedra and crossing tetrahedra that is also an inflation triangulation for some ideal triangulation $\mathcal T^*$. Given an inflation triangulation $\mathcal T$, define the \textit{length of the inflation} by $len(\mathcal T) = \beta +2\gamma$ where $\beta$ is the number of band tetrahedra and $\gamma$ is the number of crossing tetrahedra. 

A common set of combinatorial transformations of triangulations are the \textit{Pachner moves} (\textit{bistellar flips}) \cite{Pachner1978}. Let $\Delta_i, \Delta_j$ be two distinct tetrahedra identified at face $\sigma$ in $\mathcal T$. A \textit{2-3 move at} $\sigma$ is performed at the \textit{bi-pyramid}, $P$, formed by the union of the two tetrahedra $\widetilde\Delta_i,\widetilde\Delta_j$ by removing $\tilde \sigma$ and the two tetrahedra and replacing them by the triangulation of the bi-pyramid given by three tetrahedra $\widetilde{\Delta}'_1,\widetilde{\Delta}'_2,\widetilde{\Delta}'_3$ arranged around a new edge having endpoints the vertices of $P$ not contained in $\sigma$, giving a new triangulation $\mathcal T'$ of $M$. If $\widetilde\Delta_i$ and $\widetilde \Delta_j$ have unglued faces which share an edge in the material boundary of $\mathcal T$ then the bi-pyramid is referred to as a \textit{square pyramid}, the two unglued faces being the \textit{square base}. A \textit{2-2 move} can be performed on a square pyramid by removing $\widetilde\Delta_i,\widetilde \Delta_j$ and replacing them with a triangulation $\widetilde\Delta_i',\widetilde \Delta_j'$ of the square pyramid characterized by changing the diagonal on the the square base.

\begin{remark}\label{23normal}
	Let $\mathcal T$ be a triangulation of $M$, and $\mathcal T'$ a triangulation of $M$ obtained after some 2-3 move on $\mathcal{T}$. If $S$ is a normal surface in $\mathcal T$, then there exists a normal surface $S'$ in $\mathcal T'$ such that $S$ is isotopic to $S'$ as properly embedded surfaces in $M$. Similarly, if the cell-decomposition on $S$ does not include both of the two normal quadrilateral types that separate the diagonal of the base of a square pyramid, then after a 2-2 move on the square pyramid there is again a normal surface $S'$ in $\mathcal T'$ isotopic to $S$. Further if $\Delta, \Delta'$ are corresponding tetrahedra in $\mathcal T, \mathcal T'$ unaffected by the Pachner move, then the normal discs of $S,S'$ in $\Delta,\Delta'$ are in 1-1 correspondence.  If $S$ meets the bi-pyramid $P$ in a disc formed by a normal quadrilateral as in the left of Figure \ref{fig:23effect}, then after a 2-3 move the disc is divided into two normal quadrilaterals and one normal triangle in $S'$. If $S$ meets $P$ in a disc formed by two normal quadrilaterals as in the right of Figure \ref{fig:23effect}, then after a 2-3 move the disc is divided into two normal triangles and one normal quadrilateral in $S'$.
\end{remark}

\begin{figure}
	\includegraphics[width=5.8in]{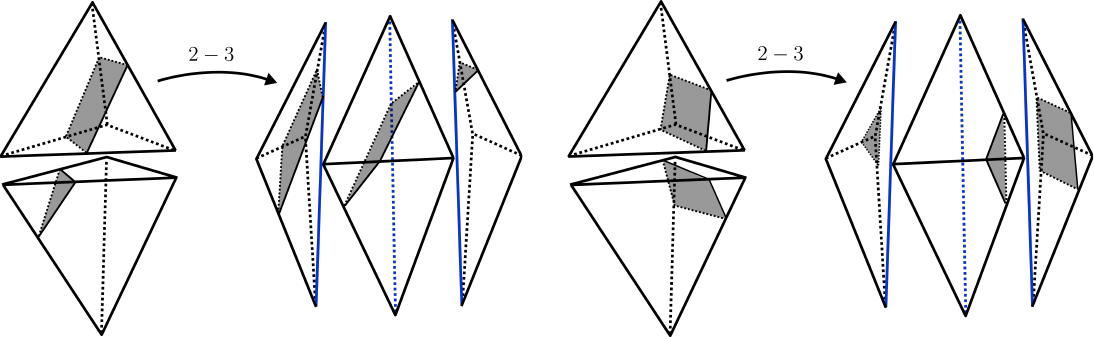}\caption{Shown are subcomplexes of a normal surface meeting a bi-pyramid before and after a 2-3 move.}\label{fig:23effect}
\end{figure}

By Remark \ref{23normal} any such triangulation $\mathcal T'$ obtained by a 2-3 move on an inflation triangulation has normal boundary. We wish to consider two particular 2-3 moves:  \textit{shortening moves}, a 2-3 move performed at a face corresponding to a trapezoid in $\mathcal C_X$ that is common to two band tetrahedra or shared by a crossing tetrahedron and a band tetrahedron; and \textit{site-swapping moves} (or simply \textit{site-swaps}), a 2-3 move performed at a face corresponding to a trapezoid common to a band tetrahedron and a branch-point tetrahedron.

\begin{proposition}\label{shortin}
	Let $\mathcal T$ be an inflation triangulation of $M$, $\mathcal T'$ a triangulation obtained by performing a shortening move. Then $\mathcal T'$ is an inflation triangulation. Further, $len(\mathcal T') \leq len(\mathcal T)$.
\end{proposition}

\begin{proof}
	Let $M$ and $\mathcal T$ be given and let $\beta=|\{\widetilde b_i\}|$ be the number of band tetrahedra, $\gamma=|\{\widetilde c_j\}|$ be the number of crossing tetrahedra. Let $B$ be the boundary-linking surface and let $\mathcal C_X$ be the induced cell-decomposition of $X$ the closure of the component of the complement of $B$ not containing the vertex of $\mathcal T$. After a 2-3 move the normal surface $B$ will be identified with an isotopic normal surface which we will also refer to as $B'$. We proceed by giving the case for each class of shortening move:
	
	\begin{mycase}
		
		\case\textbf{(between a band tetrahedron and a crossing tetrahedron)}
		Let $b_i,c_j$ be a band tetrahedron and a crossing tetrahedron respectively, identified at face $\sigma$, the lift of which contains a trapezoidal face of $\mathcal C_X$. There are two distinct edges $\tilde e_1 = \langle ab\rangle, \tilde e_2=\langle cd\rangle$ in the lift $\tilde c_j$ that lie in the material boundary of $\mathcal T$. One of these edges, say $\tilde e_1$, also lies in $\tilde b_i$. Consider the lifts of the intersection of $B$ with the bipyramid $P$ formed from $\tilde b_i,\tilde c_j$. After performing a 2-3 move at $\sigma$, the unique tetrahedron $\widetilde{\Delta}_{\tilde e_1}$ containing $\tilde e_1$ contains a normal quadrilateral of $B'$ separating $\tilde e_1$. For the two vertices of $\widetilde{\Delta}_{\tilde e_1}$ not in the closure of $\tilde e_1$ the normal triangle that separates the respective vertex is in $B'$. The other two tetrahedra have the edge $\tilde e_2$ in common. In each, there is a single quadrilateral separating $\tilde e_2$, and for the appropriate vertices not in the closure of $\tilde e_2$ a normal triangle separating that vertex. Thus the intersection of each tetrahedra of $B'$ with $P$ forms a truncated prism in the induced cell-decomposition $\mathcal C'_X$.  As $\mathcal C_X$ is simple, by Remarks \ref{23normal} and \ref{naming} the  cell-decomposition $\mathcal C'_X$ induced on $X$ by $\mathcal T'$ is simple. There is a trapezoidal face in the interior of $P$, the closure of which is connected; the preimage of all other components $\mathbb P'_i$ of $\mathbb P(\mathcal C'_X)$ are a component $\mathbb P_j$ of $\mathbb P(\mathcal C_X)$. Thus $\mathbb P(\mathcal C'_X)$ is a trivial combinatorial product. By Theorem \ref{FTC}, $\mathcal T'$ is an inflation triangulation of $M$. The number of band tetrahedra $\beta'=\beta+2$, the number of crossing tetrahedra $\gamma'=\gamma-1$; hence $len(\mathcal T')=len(\mathcal T)$.

		\case \textbf{(between band tetrahedra)}
		Let $b_i, b_j$ be two band tetrahedra identified at face $\sigma$ containing a trapezoidal face of $\mathcal C_X$. Each vertex of the associated bi-pyramid $P$ gets identified to the sole vertex of $\mathcal T$ and a single edge in $P$ is the lift of some edge, $e$ in the material boundary of $\mathcal T$.  Consider the lifts of the intersection of the boundary-linking surface $B$ with $P$. After the 2-3 move at $\sigma$, suppose $\widetilde\Delta_{\tilde e}$ is the unique tetrahedra in $P$ containing an edge, $\tilde e$, in the lift of $e$. The associated normal discs of $B'$ in $\widetilde{\Delta}_{\tilde e}$ are then a normal quadrilateral separating $\tilde e$ and two normal triangles separating the two vertices not in $\tilde e_k$. The associated normal discs in the remaining two tetrahedra  $\widetilde\Delta_2, \widetilde\Delta_3$ of $P$ are one of each of the four normal classes of normal triangle discs in each respective tetrahedra. As $\mathcal C_X$ is simple, by Remarks \ref{23normal} and \ref{naming} the  cell-decomposition $\mathcal C'_X$ induced on $X$ by $\mathcal T'$ is simple. Further the preimage of a component $\mathbb P'_i$ of $\mathbb P(\mathcal C'_X)$ is a component $\mathbb P_j$ of $\mathbb P(\mathcal C_X)$. Thus $\mathbb P(\mathcal C'_X)$ is a trivial combinatorial product. By Theorem \ref{FTC}, $\mathcal T'$ is an inflation triangulation of $M$. The number of band tetrahedra $\beta'=\beta-1$, the number of crossing tetrahedra $\gamma'=\gamma$; hence $len(\mathcal T')=len(\mathcal T)-1$.
		
\end{mycase}\end{proof}

\begin{lemma}\label{shorty}
	Let $M$ be an orientable, compact, irreducible, $\partial$-irreducible 3-manifold with non-empty connected boundary homeomorphic to a torus. There exists an inflation triangulation $\mathcal T$ of $M$ having length 2. Further, there is an algorithm to construct such a $\mathcal T$ from any given ideal triangulation $\mathcal T^*$.
\end{lemma}

\begin{proof}
	Let such a 3-manifold $M$ be given, and let $\mathcal T^*$ be an ideal triangulation of $M$. By \cite{Bing, Jaco2003, Moise} there exists an algorithm to construct an ideal triangulation of $M$. Let $\mathcal T_\xi$ be an inflation of $\mathcal T^*$ constructed as in \cite{Jaco14} by way of the frame $\xi$. By Case 1 of Proposition \ref{shortin}, we may form an inflation triangulation $\mathcal T'$ having no crossing tetrahedra. As $\xi$ is a frame and $\partial M$ is a torus, $\xi$ must consist of exactly two or three branches. 
	
	First, assume $\mathcal T'$ has two branches. Then repeated application of Case 2 of Proposition \ref{shortin} reduces the number of band tetrahedra to two and the conclusion of the lemma holds.
	
	Now, assume $\mathcal T'$ has three branches. By repeated application of Case 2 of Proposition \ref{shortin}, we construct an inflation triangulation $\mathcal T''$ having three branches, each containing one band tetrahedron. Choose $b_i$, a band tetrahedron. Note each face of $b_i$, corresponding to a trapezoid of the induced cell-decomposition $\mathcal C_X$ on the complement of the boundary-linking surface $S'$ in $\mathcal T'$ is identified to one of, $p_j,p_k$, the two branch-point tetrahedra. Let $\sigma$ be the common face of $b_i$ and $p_j$ and consider a site-swap at $\sigma$. Let $p'$ be the tetrahedron of $P$ after the 2-3 move  which contains a face, $\tau$, included in the material boundary of $\mathcal T''$. The boundary-linking surface, $S''$ in $\mathcal T''$, meets $\widetilde p'$ in two normal triangles parallel to $\tau$, thus $\widetilde p'$ is a branch-point tetrahedron. The face of $P$ in $\tilde p'$ that corresponds to the face in $\tilde b$ identified to $\tilde p_k$ remains identified to $\tilde p_k$. This implies $\mathcal T''$ is an inflation triangulation with two branches. Again, repeated application of Case 2 of Proposition \ref{shortin} reduces the number of band tetrahedra to two completing the proof.
\end{proof}

We will now establish a relationship between normal surfaces with boundary in a short inflation $\mathcal T$ and spun-normal surfaces in the ideal triangulation obtained by crushing $\mathcal T$ along $\partial M$. First, we note as $\mathcal T$ is a minimal-vertex triangulation of $M$, $\partial M$ has an induced triangulation given by the two unglued faces of $\mathcal T$. For a normal surface $S$ with nonempty boundary, $\partial S$ has an induced cell-decomposition consisting of normal arcs in the material boundary of $\mathcal T$. If $S$ is incompressible and $\partial$-incompressible, then each component of $\partial S$ is non-trivial in $\pi_1(\partial M)$. The topology of $\partial M$ then imposes a restriction on the normal arcs of $\partial S$.

\begin{lemma}[Theorem 3.6 of \cite{Jaco98}]\label{lem:arc}
	Let $\mathcal T$ be the triangulation of a the torus $\mathbb T^2$ consisting of two triangles. Let $s$ be a connected simple closed curve on $\mathbb T^2$ that is not nullhomoptic. Then $s$ is isotopic to a normal curve such that for each pair $(\alpha,\alpha'),(\beta,\beta'),(\gamma,\gamma')$ of normal arcs as shown in Figure \ref{fig:torusarcs}, the number of arcs contained in the induced cell-decomposition of $s$ is of the form $(n,n),(k,k),(l,l)$ for $n,k,l\in \mathbb Z_{\geq 0}$ with at least one of $n,k,l=0$.
\end{lemma}

\begin{figure}
	\centering
	\includegraphics[width=2in]{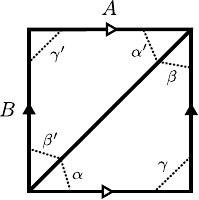}
	\caption{Normal arc classes on the minimal vertex triangulation of the torus.}
	\label{fig:torusarcs}
\end{figure}

For a normal surface $S$ in $\mathcal T$ an inflation triangulation, let $x_\alpha,y_\alpha$be represent the number of normal triangles and normal quadrilaterals respectively in the induced cell-decomposition of $S$, where $\alpha$ is a normal arc in the material boundary of $\mathcal T$ as per Figure \ref{fig:torusarcs}. Likewise define $x_{\alpha'},y_{\alpha'},x_{\beta},y_\beta,\dots$. By Lemma \ref{lem:arc} we have a system of linear equations that further restricts the matching equations:
\begin{equation}\label{toruseq}
	\begin{split}
		x_{\alpha}+y_{\alpha} & = x_{\alpha'}+y_{\alpha'} \\
		x_{\beta}+y_{\beta} & = x_{\beta'}+y_{\beta'} \\
		x_{\gamma}+y_{\gamma} & = x_{\gamma'}+y_{\gamma'}
	\end{split}
\end{equation}

We refer to Equations \ref{toruseq} as the \textit{boundary equations}. Lemma \ref{lem:arc} further implies that if $S$ is incompressible and $\partial$-incompressible then at least one of the above equations is identically zero.

For an inflation triangulation $\mathcal T$ with boundary-linking surface $B$, each band tetrahedron $b_i$ is characterized by containing a single normal quadrilateral disc in the cell decomposition of $B$. Similarly, a crossing tetrahedron $c_j$ contains two normal quadrilaterals from $B$. In each of $b_i$ and $c_j$ we call the remaining two normal quadrilateral types \textit{incompatible}. If a normal surface $S$ in $\mathcal T$ contains no incompatible normal quadrilaterals in its induced cell-decomposition, then we say $S$ is \textit{compatible}. 

\begin{corollary}\label{cor:onequad}
	Let $\mathcal T$ be a short inflation such that the material boundary of $\mathcal T$ is homeomorphic to a torus. Let $S$ be a normal surface with non-empty boundary in $\mathcal T$ such that no component of $\partial S$ bounds a disc in the material boundary of $\mathcal T$. If $S$ has no incompatible normal quadrilaterals in one band tetrahedron of $\mathcal T$, then $S$ is compatible.
\end{corollary}

\begin{proof}
	For a short inflation, the normal discs of a surface $S$ in the band tetrahedra together with the normal arcs of $\partial S$ fully determine the normal discs of $S$ in the branch-point tetrahedra. The result then follows immediately from the restriction that no component of $\partial S$ bounds a disc in the material boundary of $\mathcal T$ together with Lemma \ref{lem:arc}. 
\end{proof}

We now come to the main technical results of this work: the construction of a short inflation in which a surface is compatible. It is compatibility that will ensure the existence of a spun-normal surface upon crushing the boundary of $\mathcal T$.

\begin{theorem}\label{shortquads}
	Let $M$ be an orientable, compact, irreducible, $\partial$-irreducible 3-manifold with non-empty connected boundary homeomorphic to a torus. Let $S$ be a connected, properly embedded, incompressible, $\partial$-incompressible surface in $M$ with $\partial S$ connected. There exists a short inflation triangulation $\mathcal T$ of $M$ in which $S$ is isotopic to a compatible normal surface. Further, there is an algorithm to construct such a $\mathcal T$.
\end{theorem}

\begin{proof}
	Let $M$ and $S$ be given as in the supposition and let $\mathcal T$ be a short inflation constructed as in Lemma \ref{shorty}. By Theorem \ref{haken} $S$ is isotopic to a normal surface that we again refer to as $S$. If $S$ is compatible, then we are done. Suppose instead that the cell-decomposition of $S$ has an incompatible normal quadrilateral.
	
	The material boundary of $\mathcal T$ is given by the two unglued faces of the branch-point tetrahedra $p_1$ and $p_2$. There are three edges in the material boundary of $\mathcal T$, one with order 2 and the others having order 3. Orient and label the two order 3 edges $A,B$ as shown in Figure \ref{fig:pyramidarcs}. The fundamental group $\pi_1(\partial M)$ is generated by the closed curves $A$ and $B$. The boundary curve $\partial S$ can be given an orientation so that $\partial S = n A + mB$, with $\gcd(m,n)=1$. As in Lemma \ref{lem:arc} and Figure \ref{fig:torusarcs}, we refer to the six normal arc types in the material boundary of $\mathcal T$ as $\alpha,\alpha',\beta,\beta',\gamma,\gamma'$. Let $x_\alpha,y_\alpha$ count the number of normal triangles and normal quadrilaterals respectively in the cell-decomposition of $S$ with normal arc type $\alpha$; similarly, for the other normal arc types define such $x_{\alpha'},y_{\alpha'}\dots$ Note that as $S$ is $\partial$-irreducible, it cannot be the case that both $x_\gamma$ and $x_{\gamma'}$ are positive. The branch-point tetrahedra have a single face $\sigma$ in common with the lifts $\tilde p_1,\tilde p_2$ meeting at $\sigma$ forming a square pyramid $P$ with base the unglued faces of $\tilde p_1,\tilde p_2$. Up to reflective symmetry of the material boundary of $\mathcal T$ and relabeling the edges and possibly performing a 2-2 move on $P$, it can be assumed $0\leq m\leq n$. Thus we need only to prove the result when $\partial S$ has normal arc types of $\alpha,\alpha',\gamma,\gamma'$. That is, $x_{\beta},x_{\beta'},y_{\beta},y_{\beta'}=0$.	
	\begin{figure}
		\includegraphics[width=2in]{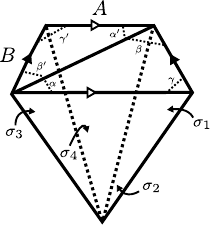}\caption{The two branch-point tetrahedra, the order 3 edges $A,B$ are labeled. On the unglued faces, the normal arc types are labeled. The faces glued to band tetrahedra are labeled in order $\sigma_1,\sigma_2,\sigma_3,\sigma_4$.}\label{fig:pyramidarcs}
	\end{figure}
	
	In the short inflation $\mathcal T$, the normal discs of $S$ in the band tetrahedra are completely determined by the normal discs in the branch-point tetrahedra. The proof then proceeds by cases according to which normal quadrilaterals (if any) are present in the branch-point tetrahedra. Figure \ref{fig:pyramidarcs} illustrates the square pyramid with normal arc types $\alpha,\alpha',\gamma,\gamma'$ and labeled faces $\sigma_1,\sigma_2,\sigma_3,\sigma_4$ at which site-swapping moves will be performed.
	
	\vspace{1em}
	
	\begin{mycase}
		\case \textbf{(no normal quadrilaterals in $P$)} 
		If $x_\alpha > 0$, then under the assumption that there are no normal quadrilaterals in $P$ we must have $x_{\beta'}=x_\alpha$. Thus we have exactly $x_\gamma = x_{\gamma'} = 1$ and each band tetrahedra has exactly one incompatible normal quadrilateral. After a site-swap at $\sigma_1$ the resulting inflation triangulation has one branch containing two band tetrahedra and another branch containing one band tetrahedron. The latter contains no incompatible quadrilaterals, thus after shortening the length 2 branch $S$ is compatible by Corollary \ref{cor:onequad}. Case 1 is then complete.
		
		\case \textbf{($y_\gamma > 0$ and  $y_{\alpha'}=0$)}
		As $y_{\alpha'} = 0$, it must be that $x_\alpha = x_{\beta'}$. Thus $x_\alpha = 0$ and by Equation \ref{toruseq} $x_{\alpha'}=0$. Further, since $\partial S$ is connected and $S$ is $\partial$-irreducible $x_\gamma = y_{\gamma'}=0$. In particular, $y_\gamma = x_{\gamma'} = 1$. This determines that there are no incompatible normal quadrilaterals in either band tetrahedron and Case 2 is done.
		
		\case \textbf{($y_{\alpha'}>0$, $y_\alpha = 0$, and $y_\gamma \geq 0$)}
		The gluing equation of the square pyramid and $y_\alpha =0$ implies that $x_{\alpha'}=0$. Then $x_\alpha = y_{\alpha'}$ and $x_{\gamma'} =  x_\gamma + y_\gamma$. If $x_\gamma = 0$ then $S$ is compatible and Case 3 is done. If $x_\gamma > 0$, then there are $x_\gamma$ incompatible normal quadrilaterals in the band tetrahedron glued to the square pyramid along faces $\sigma_1,\sigma_3$. This normal quadrilateral type is such that the normal arcs in $\sigma_1,\sigma_3$ match with normal arcs from the normal triangles having arcs in the material boundary of $\mathcal T$ of $\gamma$ and $\gamma'$ respectively. After a site-swap at $\sigma_1$, the resulting inflation has two branches one of which contains no incompatible normal quadrilaterals. Hence by Corollary \ref{cor:onequad}, after shortening the inflation $S$ is compatible. Thus completing Case 3.
		
		\case \textbf{(Analogous Cases)}
		If $y_\gamma' > 0$ and $y_{\alpha}  = 0$, then a site-swap at $\sigma_2$ produces analogous effects as the site-swap at $\sigma_1$ in Case 2. Similarly, if $y_\alpha >0, y_{\alpha'}=0, $ and $y_{\gamma'}\geq 0$. Then a site-swap at $\sigma_2$ is analogous to Case 3.
		
		\case \textbf{($y_\alpha,y_{\alpha'}>0$)}
		Recall that $\partial S$ is written as $\partial S = nA + mB$ where $A,B$ are the labeled edges in Figure \ref{fig:pyramidarcs}. Since $y_\alpha,y_{\alpha'}>0$, $n \geq 2$. Then as $\partial S$ is connected $m>0$. The only normal discs in the square pyramid $P$ that are incident to the edge $A$ are those normal triangles counted by $x_\gamma$ and $x_{\gamma'}$, hence $x_\gamma = x_{\gamma'} = m$. Let $r = y_{\alpha}=x_{\alpha'}$, then $x_\alpha =y_{\alpha'}= n - m - r$. We can then determine the normal discs in the band tetrahedron glued to $P$ at $\sigma_1$ and $\sigma_2$. Namely, there are $m$ incompatible normal quadrilaterals, $r$ normal triangles matched with the $y_{\alpha}$ normal quadrilaterals and $n-m-r$ normal triangles matching the $y_{\alpha'}$ normal quadrilaterals at $\sigma_1,\sigma_2$ respectively.
		
		Consider the square pyramid $P'$ formed by the branch-point tetrahedra in $\mathcal T'$, the inflation triangulation obtained after a site-swap at $\sigma_1$ in $\mathcal T$. After appropriately labeling the order 3 edges and normal arcs in $P'$ as above and after a reflection about the diagonal edge of the square base, the normal discs in the resulting normal surface $S'$ associated to $S$ after the site-swap are counted as follows: \begin{align*}
			x_\gamma  = r  &,\,	y_\gamma = n-m-r\\
			x_\alpha = m &,\, y_\alpha = 0 \\
			x_{\alpha'} = 0 &,\, y_{\alpha'} = m \\
			x_{\gamma'} = n-m  &,\, y_{\gamma'} = 0.
		\end{align*}After shortening the inflation $\mathcal T'$ Case 5 then reduces to Case 3 above (shortening moves do not affect the normal discs in the branch-point tetrahedra). The face of the initial square pyramid $P$ at which the second site-swap (per Case 3) is to be performed is $\sigma_4$. This completes Case 5.
		
	\end{mycase}
	
	\vspace{1em}
	
	Therefore, if $S$ a connected, properly embedded, incompressible, $\partial$-incompressible surface in $M$ with $\partial S$ connected, then the short inflation $\mathcal T$ can be altered through successive site-swapping moves and shortening moves into a short inflation in which $S$ is isotopic to a compatible normal surface.
\end{proof}

For a given $\mathcal T$ and $S$, implementation of Theorem \ref{shortquads} requires first determining the normal discs in the branch-point and band tetrahedra of $\mathcal T$. Then the appropriate succession of site-swaps and shortening moves can be chosen (determined by the cases of Theorem \ref{shortquads}) with at most two site-swaps required. The effect of each site-swap on the normal discs is illustrated in Figure \ref{fig:bmoves}. With compatibility of $S$ ensured, we now prove the relation to a spun-normal surface in $\mathring M$ isotopic to the induced embedding of $\mathring S$.

\begin{figure}
	\includegraphics[width=4in]{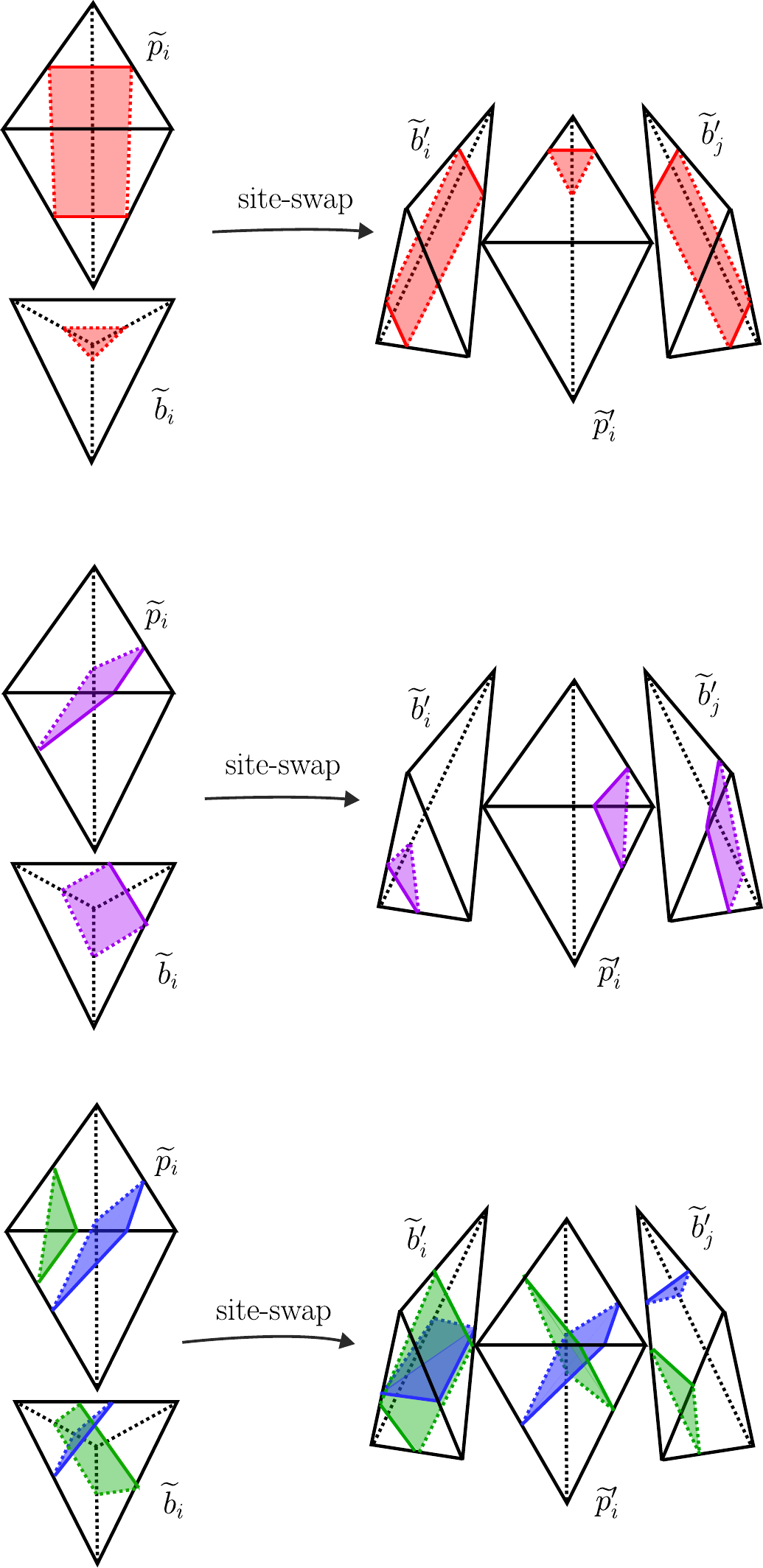}\caption{The effect of a site-swap for four possible intersections of a surface $S$ meeting at a common face of a band tetrahedron and a branch-point tetrahedron. Absent are the cases of $\widetilde b_i$ containing compatible normal quadrilaterals.}\label{fig:bmoves}
\end{figure}

\begin{theorem}\label{crushsurface}
	Let $M$ be an orientable, compact, irreducible, $\partial$-irreducible 3-manifold with non-empty connected boundary homeomorphic to a torus. Let $\mathcal T$ be a short inflation with $S$ a compatible normal surface. The ideal triangulation $\mathcal T^*$ obtained by crushing $\mathcal T$ along $\partial M$ has a spun-normal surface $S^*$ with normal quadrilaterals completely determined by the normal quadrilaterals of $S$ in $\mathcal T$, the lifts of which (relative to the crushing map) are contained in $\widetilde{\mathbf{\Delta}}^*\subset\widetilde{\mathbf{\Delta}}$. Further, $S^*$ represents the interior of a properly embedded surface in $\mathring M$ isotopic to $\mathring S$.
\end{theorem}

\begin{proof}
	Let the 3-manifold $M$, short inflation $\mathcal T = (\widetilde{\mathbf \Delta},\mathbf{\Phi})$, and compatible normal surface $S$ be given as above. Let $\mathcal T^*$ be the ideal triangulation obtained by crushing $\mathcal T$ along $\partial M$. Let $t = |\widetilde{\mathbf{\Delta}}|$. Recall that we may write $\widetilde{\mathbf \Delta}=\widetilde{\mathbf \Delta}^* \sqcup \{\widetilde b_1,\widetilde b_2\} \sqcup \{\widetilde p_1,\widetilde p_2\}$ where $\widetilde b_i$ are band tetrahedra and $\widetilde p_i$ are branch-point tetrahedra in the short inflation. Further, we have $\mathcal T^* = (\widetilde{\mathbf \Delta}^*,\mathbf{\Phi}^*)$ described by a one-to-one correspondence between the ideal tetrahedra of $\mathcal T^*$ and those in $\widetilde{\mathbf{\Delta}}^* \subset \widetilde{\mathbf \Delta}$ and $\mathbf{\Phi}^*$ induced by the crushing map. Through this one-to-one correspondence we associate the normal quadrilateral type $d^Q_i$ in $\widetilde{\Delta}_j\in\widetilde{\mathbf{\Delta}}^* \subset \widetilde{\mathbf \Delta}$ to the normal quadrilateral type $d^{Q*}_i$ in the ideal tetrahedron $\widetilde{\Delta}^*_j \subset \widetilde{\mathbf{\Delta}}^*$ of $\mathcal T^*$. From the vector $S_Q = (q_1,q_2,\dots, q_{3t})\in \mathbb R^{3t}_{\geq 0}$ representing $S$ in $\mathcal T$ with $d^Q_i$ representing a normal quadrilateral type in some $\widetilde b_1,\widetilde b_2,\widetilde p_1,\widetilde p_2$ for $i>3t-12$. We can then construct a vector $v^*_Q = (q^*_1,q^*_2,\dots, q^*_{3t-12})\in \mathbb R^{3t-12}_{\geq 0}$ with $q_i^* = q_i$ for $1\leq i \leq 3t-12$. We wish to show that $v^*_Q$ represents an admissible spun-normal surface in $\mathcal T^*$.
	
	Let $e^*$ be an edge in $\mathcal T^*$ and let $\{e_k\}_k$ be the collection of edges in $\mathcal T$ that are lifts of $e^*$ through the crushing map. If $E$ is a normal quadrilateral in some tetrahedron $\widetilde{\Delta}\in\widetilde{\mathbf{\Delta}}^*$, then the sum $\sum_k \epsilon_k$ of the senses of $E$ with respect to each $e_k$ is equal to the sense $\epsilon^*$ of the corresponding normal quadrilateral $E^*$ in $\mathcal T^*$ with respect to $e^*$. Consider the following equation that equates to zero by Equation \ref{eq:qmatch}: \begin{equation}\label{eq:totalsense}
		\sum_{i=1}^{3t} \sum_k \epsilon_{k,i} q_i=0.
	\end{equation}
	
	If each tetrahedron $\Delta$ in $\mathcal T$ with at least one edge identified to some $e_k$ lifts to a tetrahedron $\widetilde{\Delta}\in \widetilde{\mathbf{\Delta}}^*$, then $\{e_k\}_k$ is in fact a singleton and the sense for every normal quadrilateral in band and branch-point tetrahedra is zero. 
	
	There are three edges of a branch-point tetrahedron $\widetilde p_i$ which are lifts of interior edges of $\mathcal T$; if one of these three edges is in $\{e_k\}_k$, then all three are included say $e_1,e_2,e_3$. For any normal quadrilateral type in $\widetilde p_i$ we have the sum of the senses given by $\sum_{k=1}^3 \epsilon_k = 0$. 
	
	In a band tetrahedron $b_i$ there is exactly one edge, say $\langle 01 \rangle$, in the material boundary of $\mathcal T$. Let $d^Q_1$ represent the normal quadrilateral type that separates the edges $\langle 01 \rangle$ and  $\langle 23 \rangle$; let $d^Q_2,d^Q_3$ be the remaining two types.  As $S$ is compatible, the number of normal quadrilaterals in $S$ of types $d^Q_2,d^Q_3$ are $q_2=q_3=0$. The sign of a normal quadrilateral of type $d^Q_1$ with respect to the edge $\langle 23 \rangle$ is $\delta_1 =0$. If one of the (unoriented) edges $\langle 02 \rangle,\langle 03 \rangle,\langle 12 \rangle,\langle 13 \rangle$ of $\widetilde b_i$ are in $\{e_k\}_k$, then all for must be in $\{e_k\}_k$. Denote these edges by $e_1,e_2,e_3,e_4$ respectively (they must all be distinct by the inflation construction). Then for each type $d^Q_i \in \{d^Q_1,d^Q_2,d^Q_3\}$ we have $\sum_{k} \epsilon_{k,i}q_i = 0$. See Figure \ref{fig:edgesigns} for an illustration of the signs that contribute to these sense calculations.
	
	\begin{figure}
		\includegraphics[width=5.65in]{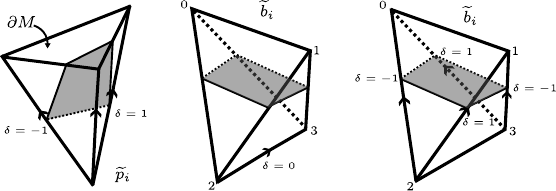}\caption{Left: the signs of a quadrilateral in $\widetilde p_i$ with respect to lifts of oriented edges $e_1,e_2,e_3$. Right: the signs for a compatible quadrilateral when $\langle 23 \rangle$ or $\langle 20 \rangle,\langle 21 \rangle,\langle 30 \rangle,\langle 31 \rangle$ are in the lift of $e$.}\label{fig:edgesigns}
	\end{figure}
	
	From the above and Equation \ref{eq:totalsense} we have \begin{align*}
		0 & = \sum_{i=1}^{3t} \sum_k \epsilon_{k,i} q_i \\
		& = \sum_{i=1}^{3t-12}\sum_k\epsilon_{k,i} q_i + \sum_{i=3t-11}^{3t} \sum_k\epsilon_{k,i} q_i \\
		& = \sum_{i=1}^{3t-12}\epsilon^*_{i} q^*_i
	\end{align*}Then $v^*_Q$ satisfies the Q-matching equations and represents an admissible solution. That is, there is some spun-normal surface $S^*$ in $\mathcal T^*$ with $q^*_1,q^*_2,\dots, q^*_{3t-12}$ counting the normal quadrilaterals of each type. By Tollefson and Kang \cite{TollefsonQ,Kang2005}, $S^*$ is the unique spun-normal surface represented by $v^*_Q$.
	
	Let $B$ be the boundary linking surface in $\mathcal T$. For a normal quadrilateral of $B$ in a band tetrahedron $\widetilde{b}_i$ there are two opposite edges $\varepsilon_1$ and $\varepsilon_2$ which crush to an edge in the frame with which $\mathcal T$ was inflated (see Figure \ref{fig:edgeinflate}). Through normal isotopies $B$ and $S$ can be positioned such that the intersection $\Gamma$ of $S$ with $B$ has no arcs in the normal quadrilaterals of $B$ that are isotopic to $\varepsilon_1,\varepsilon_2$. Crushing those cells of $\mathcal C_X$, the cell-decomposition of the complement of $B$ not containing the vertex of $\mathcal T$, which are not truncated-tetrahedra yields a cell-decomposition $\bar{\mathcal T}$ of $M$ consisting of truncated tetrahedra. By the same arguments of Lemma 3.4 in \cite{Bryant}, crushing these cells maps normal discs of $S$ to discs in $\bar{\mathcal T}$. As $S$ is compatible, the above Q-matching equations show that these discs glue together to form a surface $\bar S$ that is isotopic to $S$ in $M$. Specifically, crushing induces a cell-like map between $\bar S$ and $S \cap \mathcal C_X$. By \cite{armentrout, sieb} $\bar S$ and $S\cap \mathcal C_X$ are homeomorphic. By the preceding arguments, $\bar S$ is a core of a spun-normal surface, the normal quadrilaterals of which determined by those normal quadrilaterals of $S$ in $\widetilde{\mathbf{\Delta}}^*$. Hence $\bar S$ is a core of $S^*$ completing the proof.\end{proof}
	
	While in general for a band tetrahedron $\widetilde b_i$ the edge $\langle 23\rangle$ may be identified to the same edge class in $\mathcal T$ as one of (unoriented) $\langle 02 \rangle,\langle 03 \rangle,\langle 12 \rangle,\langle 13 \rangle$, the ideal triangulations used in Lemma \ref{shorty} are 0-efficient \cite{Jaco2003} and thus have no faces which are 3-folds or dunce hats \cite{jaco2020minimal}. Whence, in this work these three edge classes will be distinct.
	
	As incompatible normal quadrilaterals intersect the material boundary of $\mathcal T$, any closed normal surface must be compatible. This gives an alternate proof to Theorem 3.5 of \cite{Bryant} in the case of the inflation triangulation being a short inflation.

\begin{remark}\label{rem:curve}The identification of a compatible normal surface $S$ with $\bar S$ and $S^*$ allows for the boundary-slope of $S^*$ to be computed directly from $S$. In an inflation $\mathcal T,$ the edges of the square pyramid $A,B$ generate $\pi_1(\partial M)$. The homeomorphism of $M$ induced by the crushing map from $\mathcal T$ to $\bar{\mathcal T}$ sends $A,B$ to the two branches $\beta_1,\beta_2$ of the frame $\xi$. Thus the curve $\partial S = nA + mB$ is taken to $\partial \bar S = n\beta_1 + m\beta_2$.
	
\end{remark}

	\section{The Algorithm}\label{Sec:Algo}
	
		We now give an algorithm to construct, for the exterior of a knot in a rational homology 3-sphere $M_k$ that fibers over the circle with fiber $F$, an ideal triangulation for which the interior of $F$ is isotopic to a spun-normal surface. We then give a generalization for general circle bundles with fibers homeomorphic to an orientable, once-punctured surface.

Both algorithms make use of Lemma \ref{shorty}. It should be noted in general it may be possible to strategically apply Proposition \ref{shortin} to arrive at an inflation triangulation (not short) where the ideal triangulation obtained by crushing $\partial M$ has the desired property. The current state of the art to verify that the fiber is realized as a spun-normal surface is through the use of the computer software \verb*|tnorm| \cite{Tnorm} developed by Worden, this software must compute the vertex surfaces for each ideal triangulation. There exist frames for which a desired ideal triangulation is not realized at any intermediate step. One such example appears in Section \ref{Sec:Examples}, the three tetrahedra ideal triangulation of the pretzel knot $P(-2,3,7)$ complement in $S^3$.

\begin{theorem}\label{main1}
	Let $M_K$ be the exterior of a knot in a rational homology 3-sphere, and suppose $M_K$ fibers over the circle with fiber $F$. There is an algorithm to construct an ideal triangulation $\mathcal T^*$ of $M_k$ with $S$ a spun-normal surface such that $\mathring S$ is isotopic to $\mathring F$ in $\mathring M_K$. Further, the algorithm will identify $S$.
\end{theorem}

\begin{proof}
	Let $M_K,F$ be given. There is an algorithm given in \cite{Jaco2003} to construct an ideal triangulation $\mathcal T^*$ of $M_K$. Let $\xi$ be a frame of $\mathcal T^*$ and $\mathcal T_{\xi}$ the inflation of $\mathcal T^*$ along $\xi$ as per \cite{Jaco14}. By Lemma \ref{shorty}, we can construct a short inflation triangulation $\mathcal T'$ of $M$. By Lemma \ref{vertexfiber}, there exists a vertex solution, $S'$, in $\mathcal T'$ isotopic to $F$. Using Theorem \ref{shortquads}, modify $\mathcal T'$ to a short inflation $\mathcal T''=(\widetilde{\mathbf{\Delta}}'',\mathbf \Phi'')$ such that the normal surface $S''$ corresponding to $S'$ after the sequence of 2-3 moves has no normal quadrilaterals in the band tetrahedra $b''_1,b''_2$ incident to the material boundary of $\mathcal T''$. Let $\mathcal T^{**}$ be the ideal triangulation obtained by crushing $\mathcal T''$ along $\partial M$. By Theorem \ref{crushsurface}, there is a spun-normal surface, $S^*$, representing a properly embedded surface (we shall also refer to as $S^*$) in $M_K$ isotopic to $S''$. Whence $S^*$ is the desired spun-normal surface, the interior of which in $\mathring M$ is isotopic to $\mathring F$. Further, the normal quadrilaterals of $S^*$ are completely determined by the normal quadrilaterals of $S''$ contained in the tetrahedra $\widetilde{\mathbf{\Delta}}^*\subset \widetilde{\mathbf{\Delta}}''$ as denoted in Remark \ref{naming} .
\end{proof}

The key to identifying the quadrilateral coordinates for $S^*$ above lies in the fact that $S'$ is a vertex solution. While the vertex solutions can be enumerated using linear algebra, there are algorithms developed by Burton which are optimized for normal surface theory in \cite{Burton1,Burton2}.

\begin{theorem}\label{main2}
	Let $M$ be a compact, connected, irreducible, $\partial$-irreducible, atoroidal, orientable 3-manifold with non-empty connected boundary homeomorphic to a torus. Suppose $M$ fibers over the circle with fiber $F$. There is an algorithm to construct an ideal triangulation $\mathcal T^*$ of $M$ with $S$ a spun-normal surface such that $\mathring S$ is isotopic to $\mathring F$. Further the algorithm will identify $S$.
\end{theorem}

The only significant difference when compared to Theorem \ref{main1} is the inability to locate the fiber amongst the vertex solutions in the short inflation $\mathcal T'$.

\begin{proof}
	Let $M,F$ be given. There is an algorithm given in \cite{Jaco2003} to construct an ideal triangulation $\mathcal T^*$ of $M$. Let $\xi$ be a frame of $\mathcal T^*$ and $\mathcal T_\xi$ the inflation of $\mathcal T^*$ along $\xi$. By Lemma \ref{shorty}, we can construct a short inflation triangulation $\mathcal T'$ of $M$. Enumerate the vertex solutions in $\mathcal T'$. To the collection of vertex solutions $\{F_i\}_{i\in I}$, include all sums of the form $\sum_{i\in J} F_i$ where $J\subset I$ of carnality less than $\beta_1$ the first betti number of $M$. Call this collection $\mathfrak C$. By Remark \ref{thurstonvert}, there is a normal surface $S' \in \mathfrak C$ that is isotopic to $F$. The remainder of the proof follows that of Theorem \ref{main1}.
\end{proof}

In practice, to construct the initial ideal triangulation $\mathcal T^*$ we use SnapPy \cite{SnapPy} which uses a heuristic to construct a triangulation with as few tetrahedra as possible by searching the Pachner graph. When the cusped manifold is hyperbolic, SnapPy will return a geometric ideal triangulation which is always 0-efficient and $\partial$-efficient \cite{Bryant, kang2005ideal}.
	
	\section{Examples and Questions}\label{Sec:Examples}
	
		The following section contains two examples of the algorithm of Theorem \ref{main1}. A third example is given to illustrate the general absence of fibers among spun-normal surfaces. In each example, Regina \cite{regina} and SnapPy are utilized via SageMath \cite{sagemath} to construct and manipulate triangulations. Worden's program \verb*|tnorm| \cite{Tnorm} is used to verify that fibers are represented by embedded spun-normal surfaces. 

\begin{example}\label{ex:5_1}
	
	We begin with this section by applying Theorem \ref{main1} to build an ideal triangulation of the complement of the trefoil, $M_K$ which fibers over the circle with incompressible fiber a once punctured torus.

	The process begins with the ideal triangulation given by the isomorphism signature \cite{BurtonIsoSig} `cPcbbbadu' and the following gluings after orientations using the notation of Regina:
	
	\begin{table}[!h]
		\center
		\begin{tabular}{c@{\hskip 0.6in}c@{\hskip 0.6in}c@{\hskip 0.6in}c@{\hskip 0.6in}c}
			\text{tet} & (012)      & (013)      & (023)      & (123)\\
			(0)         & (1)(210) & (1)(031)     & (1)(032)     & (1)(132)\\
			(1)         & (0)(210) &  (0)(031)    &  (0)(032)    & (0)(132)
		\end{tabular}
		\caption{\label{tab:41} Ideal triangulation $\mathcal T^*$ of $M_K=S^3\backslash\{\text{trefoil knot}\}$.}
	\end{table}
	
	We examine the vertex-linking surface and choose a frame, $\xi$, for the inflation as in Figure \ref{fig:trefoillink}. We then inflate along $\xi$ yielding the following inflation triangulation:
	
	\begin{figure}
		\includegraphics[width=3in]{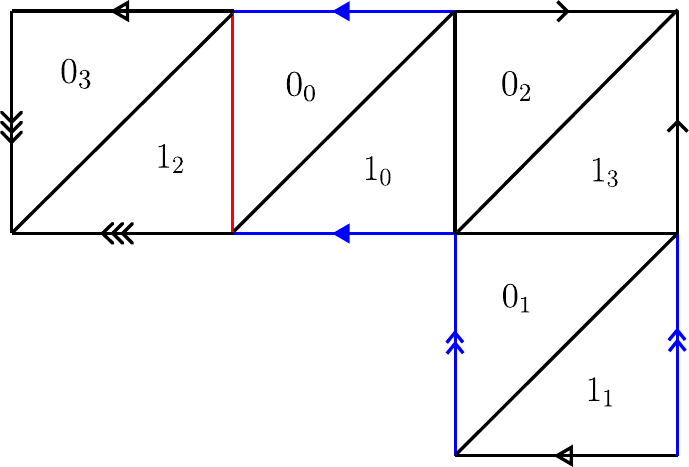}\caption{The vertex-linking surface of `cPcbbbadu,' a two tetrahedra ideal triangulation of the trefoil. Arrows denote edge identification. The labels $i_j$ correspond to the normal triangle in tetrahedron $i$ separating vertex $j$. The frame, $\xi$, consists of the red and blue edges.}\label{fig:trefoillink}
	\end{figure}
	
	\begin{table}[!h]
		\center
		\begin{tabular}{c@{\hskip 0.6in}c@{\hskip 0.6in}c@{\hskip 0.6in}c@{\hskip 0.6in}c}
			\text{tet} & (012)      & (013)      & (023)      & (123)\\
			(0)         & ($b_0$)(123) & (1)(031)     & ($b_2$)(032)     & (1)(132)\\
			(1)         & ($b_1$)(203) &  (0)(031)    &  ($b_2$)(123)    & (0)(132)\\
			($b_0$)		& ($p_1$)(123) & ($p_0$)(023) & ($b_1$)(312)	& (0)(012) \\
			($b_1$)		& ($c$)(013)	& ($c$)(321)  & (1)(102)		& ($b_0$)(230)\\
			($b_2$)		& ($c$)(320)	& ($p_0$)(213)& (0)(032)		& (1)(023)\\
			($c$)		& ($p_1$)(013)	& ($b_1$)(012)& ($b_2$)(210)	& ($b_1$)(310)\\
			($p_0$)		& 				& ($p_1$)(023)& ($b_0$)(013)	& ($b_2$)(103)\\
			($p_1$)		& 				& ($c$)(012)  & ($p_0$)(013)	& ($b_0$)(012)
		\end{tabular}
		\caption{\label{tab:42}Inflation $\mathcal T_\xi$ of $\mathcal T^*$.}
	\end{table}
	
	
	We give this triangulation to Regina and perform 2-3 moves through faces 11, 10, 15, 15 of the respective retriangulations to get the short inflation with the gluings shown in Table \ref{tab:43}.
	
	\begin{table}[!h]
		\center
		\begin{tabular}{c@{\hskip 0.6in}c@{\hskip 0.6in}c@{\hskip 0.6in}c@{\hskip 0.6in}c}
			\text{tet} & (012)      & (013)      & (023)      & (123)\\
			(0)         & 2 (123) & 1 (031)     & 8 (203)     & 1 (132)\\
			(1)         & 5 (013) &  0 (031)    &  11 (032)    & 0 (132)\\
			(2)		& 4 (123) 		& 3 (023) & 6 (320)	& 0 (012) \\
			(3)		& 				& 4 (023)  & 2 (013)		& 10 (132)\\
			(4)		& 				& 10 (310)& 3 (013)		& 2 (012)\\
			(5)		& 8 (231)		& 1 (012)& 6 (013)	& 7 (201)\\
			(6)		& 	7 (231)		& 5 (023)& 2 (320)	& 9 (201)\\
			(7)		& 	5 (231)		& 9 (231)  & 8 (021)	& 6 (201) \\
			(8)		& 7 (032)		& 11 (213)		& 0 (203)	& 5 (201)\\
			(9)		& 6 (231)		& 10 (012) 	& 11 (021)	& 7 (103)\\
			(10)	& 9 (013)		& 4 (310)		& 11 (013)	& 3 (132)\\
			(11)	& 9 (032)		&10 (023)		& 1 (032)	&  8 (103)
		\end{tabular}
		\caption{\label{tab:43} Short inflation $\mathcal T'$.}
	\end{table}
	
	
	Using Regina we now enumerate the vertex solutions. Among these surfaces there is a single essential once-punctured torus. This is surface 23 in the enumeration given by Regina version 7.3. By Lemma \ref{vertexfiber} this surface is isotopic to a fiber of the bundle structure on $M_K$. The band and branch tetrahedra are shown in Figure \ref{initialinflationTref}. After a 2-2 move at the branch-point tetrahedra $\Delta_3,\Delta_4$, we arrive at Case 3 of Theorem \ref{shortquads}. We then perform a 2-3 move at face (023) of $\Delta_2$ and subsequently shorten to arrive at the final short inflation in Table \ref{tab:43}.
	
	\begin{table}[!h]
		\center
		\begin{tabular}{c@{\hskip 0.6in}c@{\hskip 0.6in}c@{\hskip 0.6in}c@{\hskip 0.6in}c}
			\text{tet} & (012)      & (013)      & (023)      & (123)\\
			(0)         & 11 (123) & 1 (031)     & 6 (203)     & 1 (132)\\
			(1)         & 3 (013) &  0 (031)    &  8 (032)    & 0 (132)\\
			(2)		&  		& 9 (213) & 10 (013)	& 12 (310) \\
			(3)		& 	6 (231)    & 1 (012)  & 4 (013)		& 5 (201)\\
			(4)		& 	5 (231)			& 3 (023)& 10 (231)		& 7 (201)\\
			(5)		& 3 (231)		& 7 (213)  & 6 (021)	& 4 (201)\\
			(6)		& 	5 (032)		& 8 (213)  & 0 (203)	  & 3 (201)\\
			(7)		& 	5 (231)		& 9 (231)  & 8 (021)	& 6 (201) \\
			(8)		& 7 (032)		& 13 (023)	& 1 (032)	& 6 (103)\\
			(9)		& 				& 10 (012) 	& 12 (312)	& 2 (103)\\
			(10)	& 9 (013)		& 2 (023)		& 13 (312)	& 4 (302)\\
			(11)	& 7 (013)		&12 (012)		& 13 (021)	&  0 (012)\\
			(12)	& 11 (013)		& 2 (231) 	& 12 (013)	& 9 (230)\\
			(13)	& 11 (032)		& 12 (023)	& 8 (013)	& 10 (230)
		\end{tabular}
		\caption{\label{tab:44} Adjusted short inflation $\mathcal T''$ after site-swapping.}
	\end{table}
	
	\begin{figure}
		\includegraphics[width=2.5in]{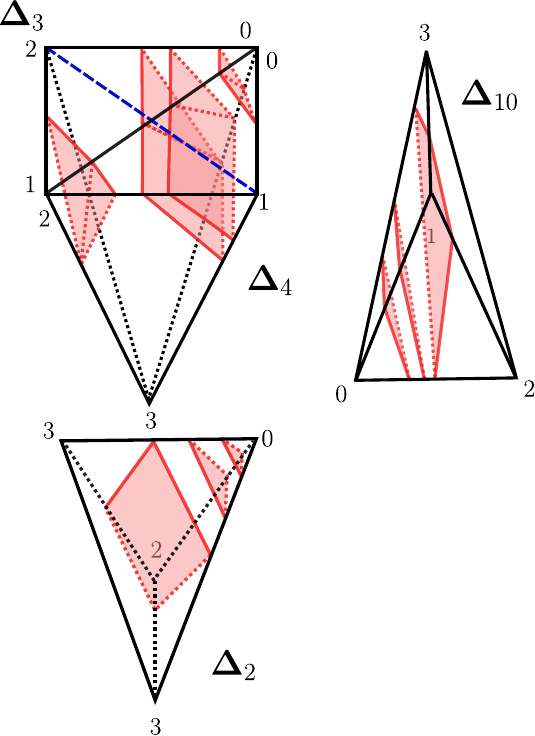}
		\caption{\label{initialinflationTref}The band and branch tetrahedra of $\mathcal T'$. Faces 3 (012) and 4 (012) constitute $\partial M_K$. The edge after a 2-2 move is shown in blue.}
	\end{figure}
	
	We again enumerate the vertex normal surfaces, finding the fiber $F$ at index 43 having no quadrilaterals in the band tetrahedra. We can now create the desired ideal triangulation of $M_K$ by crushing along the normal boundary in $\mathcal T''$. During the crushing processes as implemented by Regina, if $\Delta_i,\Delta_j\in\mathcal T''$ with $i<j$, then the corresponding ideal tetrahedra in the induced ideal triangulation $\Delta^*_l,\Delta^*_k\in\mathcal T^{**}$ are such that $l<k$. Thus we can take the quadrilateral coordinates of $F$ in $\mathcal T''$ and remove the corresponding three coordinates for each band and branch tetrahedron. We then have $F$ given as a spun-normal surface in $\mathcal T^{**}$ with Q-coordinates: $$v^*_Q=(0,1,0,0,0,1,0,0,0,0,1,0,0,1,0,0,2,0,0,1,0,0,2,0,0,0,0,1,0,0)$$. This tuple can be constructed without crushing $\mathcal T''$. Rather, take the quadrilateral coordinates of $F$ in $\mathcal T''$ and remove those twelve coordinates which correspond to normal quadrilateral types in band or branch-point tetrahedra. In Table \ref{tab:44}, these are tetrahedra numbered 2, 9, 10, and 12.
	
	
\end{example}

\begin{example}\label{ex:5_2}	We now apply Theorem \ref{main1} to the (-2,3,7)-pretzel knot $K$, K12n242 in the Hoste-Thistlewaite-Weeks enumeration \cite{Knottable}. This gives the first known example of an ideal triangulation of the knot complement $M_K$ in which a fiber of the bundle structure, $S_{5,1}$ the once punctured genus 5 surface over $S^1$ is realized as a spun-normal surface. As $M_K$ is hyperbolic, we can use \verb*|tnorm| to quickly tell if a fiber is among the spun-normal surfaces. 
	
	\begin{lstlisting}[language=Python]
		sage: import regina; import snappy; import tnorm
		sage: M=snappy.Manifold("K12n242"); T=regina.Triangulation3(M._to_string())
		sage: W1 = tnorm.load(M); B=W1.norm_ball
		sage: B.vertices()
		[Vertex 0: represented by (1/9)*S_5,1 at (-1/9,), not represented by any qtons,
		Vertex 1: represented by (1/9)*S_5,1 at (1/9,), not represented by any qtons]\end{lstlisting}As one comes to expect, the fiber is not present among the spun-normal surfaces. Table \ref{tab:45} gives the glueings for the triangulation of $M_K$, which we shall call $\mathcal T^*$. Figure \ref{fig:237vert} shows the induced triangulation of the ideal boundary and highlights the frame $\xi$.
	
	\begin{table}[!h]
		\center
		\begin{tabular}{c@{\hskip 0.6in}c@{\hskip 0.6in}c@{\hskip 0.6in}c@{\hskip 0.6in}c}
			\text{tet} & (012)      & (013)      & (023)      & (123)\\
			(0)         & 1 (032) & 1 (201)     & 2 (032)     & 1 (132)\\
			(1)         & 0 (130) &  2 (310)    &  0 (021)    & 0 (132)\\
			(2)			& 2 (231) & 1 (310)		& 0 (032)	& 2 (201)
		\end{tabular}
		\caption{\label{tab:45} Ideal triangulation $\mathcal T^*$ of $M_K$, the complement of K12n242, as generated by SnapPy.}
	\end{table}
	
	\begin{figure}
		\includegraphics[width=3in]{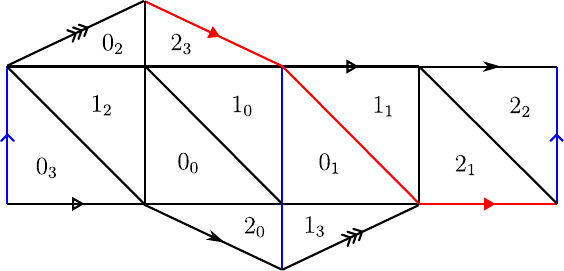}
		\caption{\label{fig:237vert}The vertex-linking surface of $\mathcal T^*$. The lables $i_j$ correspond to the normal triangle in tetrahedron $i$ separating vertex $j$. The frame, $\xi$, consists of the red and blue edges.}
	\end{figure}We now inflate $\mathcal T^*$ along $\xi$, the face identifications are given in Table \ref{tab:46}.
	
	\begin{table}[!h]
		\center
		\begin{tabular}{c@{\hskip 0.6in}c@{\hskip 0.6in}c@{\hskip 0.6in}c@{\hskip 0.6in}c}
			\text{tet} & (012)      & (013)      & (023)      & (123)\\
			(0)         & 1 (032) & $b_4$ (213)     & $b_2$ (231)     & $b_0$ (132)\\
			(1)         & $b_4$ (032) &  $b_3$ (321)    &  0 (021)    & $b_0$ (023)\\
			(2)			& $b_1$ (203) & $b_3$ (023)		& $b_2$ (203)	& $b_1$ (321) \\
			($b_0$)		& $p_1$ (123) & $c$ (012)		& 1 (123) 		& 0 (132)\\
			($b_1$)		& $c$ (013)	  & $p_0$ (023)		& 2 (102)		& 2 (321)\\
			($b_2$)		& $p_1$ (013) & $c$ (321)		& 2 (203)		& 0 (302)\\
			($b_3$)		& $c$ (320)	  & $b_4$ (012)		& 2 (013)		& 1 (310)\\
			($b_4$)		& $b_3$ (013) & $p_0$ (213)		& 1 (021)		& 0 (103)\\
			($c$)		& $b_0$ (013) & $b_1$ (012)		& $b_3$ (210)	& $b_2$ (310)\\
			($p_0$)		&			  & $p_1$ (023)		& $b_1$ (013)	& $b_4$ (103)\\
			($p_1$)		& 			  & $b_2$ (012)		& $p_0$ (013)	& $b_0$ (012)
		\end{tabular}
		\caption{\label{tab:46} Inflation $\mathcal T_\xi$ of $\mathcal T^*$ the ideal triangulation of $M_K$.}
	\end{table}
	
	Returning to SageMath, we store the triangulation $\mathcal T_\xi$ as a Regina triangulation object in the variable Tx. We then perform 2-3 moves on faces 12, 11, 12, 12, 19, 26 on respective successive triangulations. Thereafter, site-swaps and shortening moves are performed at faces 14 then 17 and again at 14 and 32. This procedure follows Case 5 of Theorem \ref{shortquads}.
	
	\begin{lstlisting}[language=Python]
		sage: shorten_path = [(2,12),(2,11),(2,12),(2,12),(2,19),(2,26),(2,14),(2,17),(2,14),(2,32)]
		sage: for dim,index in shorten_path:
		....:      Tx.pachner(Tx.face(dim,index))
		sage: Tx.size()
		21
		sage: Tx.isZeroEfficient()
		True
		sage: Normal_Surface_List = regina.NormalSurfaces.enumerate(Tx, regina.NormalCoords.NS_QUAD, regina.NormalListFlages.NS_Vertex)
		sage: for i in range(Normal_Surface_List.size()):
		....:      surface = Normal_Surface_List.surface(i)
		....:      if surface.eulerChar == 0 and not surface.hasRealBoundary():
		....:	   			ideal_tri = surface.crush()
		....:					return True
		sage: ideal_tri.removeSimplex(18); Tc.removeSimplex(17) %disconnected
		sage: M = snappy.Manifold(ideal_tri.snapPea())
		sage: W = tnorm.load(M); B = W.norm_ball
		Enumerating quad transversely oriented normal surface (qtons)... Press Ctrl+C to abort.
		progress: 100.0% ...Done.
		computing simplicial homology...Done.
		computing Thurston norm unit ball... Done.
		sage: B.vertices()
		[Vertex 0: represented by (1/9)*S_5,1 at (-1/9,), mapped from embedded qtons with index 1131,
		Vertex 1: represented by (1/9)*S_5,1 at (1/9,), mapped from embedded qtons with index 219]\end{lstlisting}
	
	Using Regina, crush Tx along the normal boundary (after finding said normal boundary via Regina), \verb*|tnorm| then identifies the vertices of Thurston's norm ball represented by embedded spun-normal surfaces in Regina's enumeration at indices 1131 and 219.
	
\end{example}

\begin{example}
	In \cite{KangNote}, Kang and Rubinstein show that for a layered triangulation of a once-punctured surface bundle, there can be no spun-normal surface isotopic to the fiber. The initial ideal triangulations in both Example \ref{ex:5_1} and Example \ref{ex:5_2} are layered triangulations \cite{agolideal, Jaco1}. We now show that the absence of spun-normal fibers is not particular to layered triangulations. In a blog post \cite{Dunfieldwordpress} Dunfield and Schliemer identified a triangulation of the cusped manifold named `v1348' in the SnapPy cusped manifold census \cite{SnapPy}. This cusped manifold is the complement of a 17 crossing knot in $S^3$ and is fibered by a once-punctured genus 5 surface. The ideal triangulation $\mathcal T^*$ generated by SnapPy for this complement is the canonical triangulation given by the Epstein-Penner decomposition \cite{EpsteinPenner}, and is notable as being the first known example of a canonical triangulation that is not a layered triangulation. Using \verb*|tnorm|, it is possible to verify that the fiber of this bundle structure is not isotopic to a spun-normal surface in $\mathcal T^*$. Via an extension of Remark \ref{rem:curve}, it was also checked that of the spun-normal surfaces with at most one almost-normal octagon type (see \cite{RubinsteinAlmostNormal,thompsonthin}) in $\mathcal T^*$, none are isotopic to the fiber.
\end{example}

\subsection{Questions}
The ideal triangulation created in Example 2 has relatively high complexity (number of ideal tetrahedra) when compared to the minimal triangulation of the complement of K12n242. There are likely lower complexity ideal triangulations in which the fiber spun-normalizes. Examining the available Pachner moves that lower complexity we find a single 3-2 move that essentially reverses the last shortening 2-3 move in \verb*|path|. The effect in terms of the spun-normal surfaces is to eliminate the possibility that the fiber spun-normalizes. Searching the Pachner graph while searching for a triangulation with a particular property is in general hard. For example, Regina's strategy for finding a triangulation with fewer tetrahedra is to randomly apply 1-4 moves and check if 3-2 moves appear.

\begin{question}
	For a fibered knot complement $M_K$, can a path in the Pachner graph be found ``efficiently" that transforms one ideal triangulation of $M_K$ into another in which a fiber is realized as an embedded spun-normal surface?
\end{question}

Though it does not occur in both Example \ref{ex:5_1} and \ref{ex:5_2}, it is possible that at an intermediate step of the shortening process crushing the boundary of $M$ yields an ideal triangulation in which the fiber is isotopic to a spun-normal surface. This leads to the following:

\begin{question}
	What criteria can be placed on a surface in an inflation triangulation that ensures an isotopic spun-normal surface exists in the ideal triangulation obtained by crushing $\partial M$?
\end{question}

For a given ideal triangulation $\mathcal T^*$ of complexity $n$, we have an upper bound on the required complexity $C$ of an ideal triangulation in which a fiber is realized as an embedded spun-normal surface, given by $C\leq n+2\text{len}(\xi)+4$. Here $\xi$ can be taken of minimal length among the frames of $\mathcal T^*$. This bound is specific to $\mathcal T^*$. Further, it is possible for $C_1>C_2$ with $n_1<n_2$. The addition of the ``4" in the upper bound is derived from the two possible site-swaps needed after the short inflation is constructed.

 In \cite{CoopTW}, an ideal triangulation with $|\mathcal T^*|=5$ is constructed for the figure-8 complement in which the fiber $S_{1,1}$ of the surface bundle structure is realized as an embedded spun-normal surface. This ideal triangulation does not come from crushing a short inflation along the boundary linking surface. That is, there is no frame of length 2 for this ideal triangulation. Further, this ideal triangulation admits a taut angle structure \cite{Lackenby}.

\begin{question}
For the complement $M_K$ of a hyperbolic knot $K$ that has the structure of a fiber bundle over a once-punctured surface $F$, does there exist an ideal triangulation admitting a strict (taut, semi) angle structure in which $F$ is isotopic to a spun-normal surface?
\end{question}

The technique used to find a short inflation in which the fiber is compatible effectively moves the incompatible normal quadrilaterals into a single branch. Corollary \ref{cor:onequad} then ensures that after shortening the fiber will be compatible. If an immersed surface meets a band tetrahedra in both incompatible normal quadrilateral types, this technique fails (see Figure \ref{fig:bmoves}). Hence the results in this work do not generally extend to virtual fibers. 

\begin{question}
	Let $M$ be a once-cusped, hyperbolic 3-manifold. Let $\widetilde{M}$ be a finite cover of $M$ that fibers over $S^1$ with fiber $F$ (not necessarily once-punctured). Does there exist an ideal triangulation of $M$ such that the image of $F$ under the covering map is regular homotopic to an immersed spun-normal surface? 
\end{question}

	\bibliographystyle{IEEEannot}
	\bibliography{bibli}
	
\end{document}